\documentclass[10pt,oneside]{amsart}

\usepackage[margin=1in]{geometry}
\usepackage{tikz}
\usepackage{enumerate}

\usepackage{amsmath,amssymb,amsthm}
\usetikzlibrary{arrows.meta}

\numberwithin{equation}{section}

\usepackage[colorlinks,citecolor=red,pagebackref,hypertexnames=false]{hyperref}

\theoremstyle{plain}
\newtheorem{theorem}{Theorem}[section]
\newtheorem{lemma}[theorem]{Lemma}
\newtheorem{corollary}[theorem]{Corollary}
\newtheorem{proposition}[theorem]{Proposition}

\theoremstyle{definition}
\newtheorem{definition}[theorem]{Definition}
\newtheorem{example}[theorem]{Example}
\newtheorem{remark}[theorem]{Remark}

\newcommand{\R}{\mathbb R}
\newcommand{\hd}{\dim_{\mathcal H}}

\title{Positive measure of unions of variable surfaces}

\author{Alex Iosevich}
\address{Department of Mathematics, University of Rochester, Rochester, NY}
\email{alex.iosevich@rochester.edu}

\author{Zhangze Li}
\address{Department of Mathematics, The Ohio State University, Columbus, OH}
\email{li.15128@osu.edu}

\author{Krystal Taylor}
\address{Department of Mathematics, The Ohio State University, Columbus, OH}
\email{taylor.2952@osu.edu}

\thanks{A.~I. was supported in part by the National Science Foundation under NSF DMS-2154232. }
\thanks{ K.T. is supported in part by the Simons Foundation Grant GR137264.}
\thanks{A.I. wishes to thank Arian Nadjimzadah for a very helpful conversation about curved Kakeya, which provided an interesting and useful perspective on the results of this paper.}

\subjclass[2020]{
Primary 42B20;
Secondary 28A78, 35S30, 58J40, 53C65
}

\keywords{
Generalized Radon transforms,
Fourier integral operators,
unions of hypersurfaces,
Kakeya phenomena,
rotational curvature,
variable coefficient averaging operators
}

\begin{document}

\begin{abstract}
Let $E \subset \mathbb R^d$, $d \ge 2$, be compact, and let $\phi(x,y)$ be a smooth function satisfying the Phong--Stein rotational curvature condition on $\{\phi(x,y)=1\}$. We prove that if $\dim_{\mathcal H}(E)>1$, then
$$
\left|\bigcup_{x \in E} \{y : \phi(x,y)=1\}\right|>0.
$$
This extends the positivity theorem of Mitsis ($d\geq3$) and Wolff ($d=2$) for spheres to a general variable coefficient setting via $L^2$ estimates for Fourier integral operators. The argument also shows that positivity is stable under finite-order degeneracies of the Monge--Amp\`ere determinant through the weighted averaging theory of Sogge and Stein.

We next consider variable level sets
$$
\Sigma_x=\{y:\phi(x,y)=t(x)\},
$$
where $t(x)$ is measurable. A maximal operator argument yields positivity under the condition $\dim_{\mathcal H}(E)>2$. We show that this loss reflects a genuine geometric obstruction related to Kakeya-type compression phenomena. In contrast, under a direct geometric intersection hypothesis controlling overlaps of the hypersurfaces $\Sigma_x$, we recover the full threshold $\dim_{\mathcal H}(E)>1$ for arbitrary measurable selections $t=t(x)$.

At the endpoint $\dim_{\mathcal H}(E)=1$, we obtain positivity under the additional assumption that $E$ is $1$-rectifiable with $\mathcal H^1(E)>0$. We also show that positivity of Lebesgue measure does not in general imply interior regularity: even for large or rectifiable parameter sets, the resulting unions may have empty interior. Finally, we discuss extensions to higher co-dimension families and the role of geometric structure in preventing compression phenomena.
\end{abstract}

\maketitle

\setcounter{tocdepth}{1}
\tableofcontents

\section{Introduction}

A recurring principle in geometric measure theory and harmonic analysis is that analytic smoothing forces geometric spreading. In problems involving unions of hypersurfaces, Fourier integral operator estimates prevent mass from concentrating excessively, and this often forces the associated geometric family to occupy a set of positive Lebesgue measure. The extent to which this spreading mechanism survives under variable coefficients, measurable parameter dependence, and degeneracies of curvature is closely connected to Kakeya-type compression phenomena and to the geometry of the underlying canonical relation.

The purpose of this paper is to isolate and develop this principle in the setting of generalized Radon transforms. We study unions of hypersurfaces of the form
$$
\{y:\phi(x,y)=t\},
$$
where $\phi$ satisfies a rotational curvature condition in the sense of Phong and Stein, and show that positivity of Lebesgue measure emerges from a balance between Fourier integral operator smoothing and geometric nonconcentration.

A prototypical result in this direction states that if $E \subset \mathbb R^d$ has Hausdorff dimension greater than $1$, then the union of unit spheres centered at points of $E$ has positive Lebesgue measure; this result follows from the work of Mitsis for $d\geq3$ and Wolff when $d=2$ \cite{Mitsis, WolffLocalSmoothing}. 
These authors also considered spheres with variable radius, showing that the same positivity conclusion holds when the radius is allowed to depend measurably on the center, provided it remains in a fixed compact interval. This shows that, in the translation-invariant setting of spheres, the positivity mechanism is robust under arbitrary measurable parameter dependence. The threshold $1$ is natural from a dimensional point of view. If $\dim_{\mathcal H}(E)<1$, the same parametrization argument shows that the union has Hausdorff dimension at most $\dim_{\mathcal H}(E)+(d-1)$, and in particular cannot have positive $d$-dimensional Lebesgue measure. Indeed, after local parametrization, the union of hypersurfaces may be parametrized as the image of a set of dimension at most $\dim_{\mathcal H}(E)+(d-1)$ under a Lipschitz map. 
Thus if $\dim_{\mathcal H}(E)<1$, one cannot expect the union to have positive $d$-dimensional Lebesgue measure. At the critical value $\dim_{\mathcal H}(E)=1$, this dimensional obstruction disappears, and additional geometric structure becomes relevant. This transition from dimensional to geometric control is one of the central themes of the present paper. Above the threshold, positivity follows from smoothing and transversality mechanisms alone, while at the critical exponent the behavior depends more delicately on the geometry of the parameter set. As shown below, rectifiability provides sufficient structure to recover positivity at the endpoint. These results illustrate that dimension alone does not control the behavior at the critical exponent, and that additional geometric structure plays an essential role.

However, the examples at the end of the paper demonstrate that without such structure, positivity can fail even for large parameter sets. In the case of circles in the plane with variable radius, a construction of Talagrand \cite{Talagrand} shows that a set of Lebesgue measure zero may contain a circle centered at every point of a line. In a different direction, Besicovitch-type constructions \cite{Besicovitch} show that for families of lines of the form
$$
\{y \in \mathbb R^2 : x \cdot y = t(x)\},
$$
with $t$ an arbitrary measurable function, one may arrange that the union has Lebesgue measure zero even when the set of parameters $E$ has positive two-dimensional Lebesgue measure. These examples indicate that in variable-coefficient settings, even large parameter sets do not guarantee positivity of measure. From a geometric point of view, the underlying issue is failure of transverse spreading: although the parameter set may be large, the associated hypersurfaces can still concentrate along lower-dimensional configurations if the family is allowed to vary too freely. This is the mechanism behind both classical Kakeya phenomena and the variable-parameter obstructions discussed in the present paper.

The results of the present paper show that positivity emerges from a balance between smoothing and geometry. Fourier integral operator estimates provide analytic spreading mechanisms, while additional geometric hypotheses prevent concentration and compression effects that can otherwise destroy positivity even in highly curved settings.

\begin{remark}
The results above are closely related to variable coefficient Falconer distance problems. If $E\subset \mathbb R^d$ is compact and $\mu$ is a Frostman measure on $E$, one may define a pushforward measure $\gamma$ on
$$
\Delta_\phi(E)=\{\phi(x,y):x,y\in E\}
$$
by
$$
\int g(t)\,d\gamma(t)=\int\int g(\phi(x,y))\,d\mu(x)d\mu(y).
$$
The density of $\gamma$, when it exists, is formally given by
$$
\gamma(t)=\int T_t\mu(x)\,d\mu(x),
$$
where $T_t$ is the generalized Radon transform associated with the level sets $\{\phi(x,y)=t\}$. Thus the Falconer-type problem and the positivity problem studied in the present paper are governed by the same family of Fourier integral operators. In both settings, the underlying mechanism is that smoothing estimates for generalized Radon transforms force geometric spreading of mass across the associated family of hypersurfaces. The present paper may therefore be viewed as part of a broader principle connecting Fourier integral operator smoothing to geometric coverage phenomena.

In the case when $\phi$ satisfies the Phong--Stein rotational curvature condition, it was shown in \cite{EswarathasanIosevichTaylor} that if $\dim_{\mathcal H}(E)>\frac{d+1}{2}$, then $|\Delta_\phi(E)|>0$. The distinction is that the Falconer problem tests the operator against two copies of the fractal measure and leads to the threshold $\frac{d+1}{2}$, whereas the union problem considered here tests against one fractal parameter set and yields the threshold $\dim_{\mathcal H}(E)>1$.

We also note that the argument in \cite{EswarathasanIosevichTaylor} extends to a broader class of phase functions than is usually stated. Indeed, the proof of the variable coefficient Falconer result relies only on $L^2$ Sobolev bounds for the associated generalized Radon transform. By the theorem of Sogge and Stein \cite{SoggeStein2}, these bounds continue to hold under the weaker assumption that the Monge--Amp\`ere determinant does not vanish to infinite order, after insertion of a suitable weight. It follows that the positivity of the distance set $|\Delta_\phi(E)|$ for $\dim_{\mathcal H}(E)>\frac{d+1}{2}$ remains valid under this weaker hypothesis.

This perspective also clarifies the relationship with earlier work of the first author and {\L}aba \cite{IosevichLaba2003}, where distance set results were obtained under assumptions on the decay of the Fourier transform of surface measures associated to convex bodies. The weighted formulation described above shows that such decay assumptions may be replaced by the weaker requirement that the Monge--Amp\`ere determinant does not vanish to infinite order. In this sense, the arguments in \cite{IosevichLaba2003} can be strengthened by incorporating the weighted averaging operators arising from the Sogge--Stein theory.
\end{remark}

\section{Preliminaries}

We first state some notation and definitions used throughout the paper.
\begin{definition}[Phong--Stein Rotational Curvature Condition]{}\label{def:phongstein}
We say a smooth function $\phi:\mathbb R^d\times\mathbb R^d\to\mathbb R$ satisfies the \textit{Phong--Stein rotational curvature condition} at $t$ if
$$
\det
\begin{pmatrix}
0 & \nabla_x \phi \\
-(\nabla_y \phi)^T & \partial^2_{xy}\phi
\end{pmatrix}
\neq 0
$$
on the set $\{(x,y):\phi(x,y)=t\}$.
\end{definition}

\begin{definition}[Radon Transform]\label{defn: Radon Trans}
    Let $\phi:\mathbb R^d\times \mathbb R^d\to\mathbb R$ be a smooth function. Given $f:\mathbb R^d\to \mathbb R$, define the \textit{Radon transform} of $f$ associated to $\phi$ at $t$ as
    $$
    R_tf(x):=\int_{\phi(x,y)=t} f(y)\psi(x,y)\,d\sigma_x(y),
    $$
    where $\psi$ is a smooth cut-off function and $\sigma_x$ is the surface measure on the set $\{y:\phi(x,y)=t\}$.
\end{definition}
When $t$ is clear from the context, we simplify the notation by writing $R_tf=Rf$.

Under the rotational curvature condition, we have the following theorem on the boundedness of the Radon transform.
\begin{theorem}[\cite{PS,PhongStein}]{}
    If $\phi\in C^\infty(\mathbb R^d\times\mathbb R^d), d\geq 2$ satisfies the rotational curvature condition at $t$, then 
    \[R_t:L^2(\mathbb R^d)\to W^{\frac{d-1}{2},2}(\mathbb R^d)\]
    is a bounded operator. Here $W^{s,2}(\mathbb R^d)$ denotes the $L^2$--Sobolev space of functions with $s$ generalized derivatives in $L^2(\mathbb R^d)$.
\end{theorem}

We also consider the maximal variant of Radon transform.
\begin{definition}[Maximal Operator]{}
    Let $\phi:\mathbb R^d\times \mathbb R^d\to\mathbb R$ be a smooth function and $I\subset(0,\infty)$ be a fixed compact interval. Given $f:\mathbb R^d\to \mathbb R$, define the \textit{maximal operator} of $f$ associated to $\phi$ as
    \[\mathcal{M}R_If(x):=\sup_{t\in I}|R_tf(x)|.\]
\end{definition}

\begin{theorem}[\cite{SoggeStein2,Stein}]{}
    Let $I\subset(0,\infty)$ be a fixed compact interval. If $\phi\in C^\infty(\mathbb R^d\times\mathbb R^d),d\geq 3$ satisfies the rotational curvature condition for all $t\in I$, then for every $\epsilon>0$,
    \[\mathcal{M}R_I:L^2(\mathbb R^d)\to W^{\frac{d-2}{2}-\varepsilon,2}(\mathbb R^d)\]
    is a bounded operator.
\end{theorem}

Finally, we recall Frostman's lemma, the definition of a $1$-rectifiable set, and the area formula for Lipschitz functions. Let $\hd$ denote the Hausdorff dimension and $\mathcal{H}^s$ denote the $s$-dimensional Hausdorff measure.
\begin{proposition}[Frostman's Lemma]{}
    Let $E\subset\mathbb R^d$ be a Borel set. Then $\mathcal{H}^s(E)>0$ if and only if there exists a compactly supported Borel measure $\mu$ such that $\operatorname{supp}\mu\subset E$ and 
    \[\mu(B(x,r))\lesssim r^s,\quad\forall\,x\in\mathbb R^d,r>0.\]
\end{proposition}

\begin{definition}[Rectifiability]{}
    A set $E\subset\mathbb R^d$ is called \textit{$1$-rectifiable} if there exists a countable collection of subsets $\{A_n\}_n$ of $\mathbb R$ and Lipschitz functions $f_n:A_n\to\mathbb R^d$ such that
    \[\mathcal{H}^1\left(E\setminus\bigcup\limits_{n=1}^\infty f_n(A_n)\right)=0.\]
\end{definition}

\begin{proposition}[Area Formula]{}
    Let $f:\mathbb R^m\to\mathbb R^d,m\leq d$ be a Lipschitz function. Then for every Borel set $A\subset \mathbb R^m$,
    \[\int_A J_f(x)\,dx=\int_{f(A)} N(f,A,y)\,d\mathcal{H}^m(y),\]
    where $J_f(x)=\det(Df(x)^TDf(x))^{1/2}$ and $N(f,A,y)=\mathcal{H}^0(A\cap f^{-1}(y))$.
\end{proposition}

\section{The Fixed Level Result}

\begin{theorem}[]{}\label{thm:fixed level}
Suppose $\phi\in C^\infty(\mathbb R^d\times\mathbb R^d),d\geq 2$ satisfies the rotational curvature condition at $1$. If $E\subset\mathbb R^d$ is a compact set with $\dim_{\mathcal H}(E)>1$, then
$$
\left| \bigcup_{x \in E} \{y:\phi(x,y)=1\} \right|>0.
$$
\end{theorem}

\begin{proof}
The idea of the proof is to turn the union of hypersurfaces into an incidence measure. This measure records how often a point $y$ lies on a hypersurface of the form $\{y:\phi(x,y)=1\}$ with $x\in E$. If the union were too small, this incidence measure would have to concentrate on a set of small Lebesgue measure. The purpose of the argument below is to show that such concentration is impossible: after mollification, the incidence measure has uniformly bounded $L^2$ norm. This forces the measure to spread over a set of positive Lebesgue measure.

We begin with a Littlewood--Paley decomposition. Let $\eta_0 \in C_c^\infty(\mathbb R^d)$ be radial with
$1_{B(0,2)}\le\eta_0\le 1_{B(0,4)}$. Let $\eta\in C_c^\infty(\mathbb R^d)$ be supported where $\frac{1}{2}\le |\xi|\le 4$ and satisfy
$$
\eta_0(\xi)+\sum_{j=1}^\infty \eta_j(\xi)=1,
$$
where $\eta_j(\xi)=\eta(2^{-j}\xi)$ for all $j\ge 1$.

Define
$$
\widehat{P_j f}(\xi)=\eta_j(\xi)\widehat{f}(\xi), \quad j\ge1,
$$
and
$$
\widehat{P_0 f}(\xi)=\eta_0(\xi)\widehat{f}(\xi).
$$

Then
$$
f=\sum_{j=0}^\infty P_j f.
$$

Since $\hd E>1$, there exists a compactly supported measure $\mu$ with $\operatorname{supp}\mu\subset E$ and exponent $1<a<\hd E$, so
$$
\mu(B(x,r))\lesssim r^a,\quad\forall\,x\in \mathbb R^d,r>0.
$$
Define $\mu^j$ as
$$
\widehat{\mu^j}=\eta_j \widehat{\mu}.
$$

\begin{lemma}
For all $j\ge 0$,
$$
\|\mu^j\|_2 \lesssim 2^{j\frac{d-a}{2}}.
$$
\end{lemma}

Since $\mu^j$ is obtained by applying a smooth Fourier multiplier to $\mu$, it is absolutely continuous with respect to Lebesgue measure and belongs to $L^2(\mathbb R^d)$.

\begin{proof}
For $j\ge 1$,
$$
\|\mu^j\|_2^2 = \int |\widehat{\mu}(\xi)|^2 \eta_j(\xi)^2 d\xi.
$$

Let $\varphi$ be supported in an annulus and equal to $1$ on the support of $\eta$. Then
$$
\eta_j^2(\xi) \le \varphi(2^{-j}\xi),
$$
so
$$
\|\mu^j\|_2^2
\le  \int |\widehat{\mu}(\xi)|^2 \varphi(2^{-j}\xi) d\xi.
$$

By Fourier inversion,
$$
\|\mu^j\|_2^2
\le2^{jd} \iint \widehat{\varphi}(2^j(x-y)) d\mu(x)d\mu(y).
$$

Since $\widehat{\varphi}$ decays rapidly, there exists $M>a$ such that
$$
|\widehat{\varphi}(2^j(x-y))|\le C_M(1+2^j|x-y|)^{-M}.
$$

Thus
$$
\|\mu^j\|_2^2
\le C_M 2^{jd} \iint (1+2^j|x-y|)^{-M} d\mu(x)d\mu(y).
$$

\begin{align*}
    \|\mu^j\|_2^2&\le C_M 2^{jd} \iint (1+2^j|x-y|)^{-M} d\mu(x)d\mu(y)\\
    &=C_M 2^{jd}\iint_{|x-y|\le 2^{-j}} (1+2^j|x-y|)^{-M} d\mu(x)d\mu(y)\\
    &+C_M 2^{jd}\sum_{l=0}^\infty\iint_{2^l<2^j|x-y|\le 2^{l+1}} (1+2^j|x-y|)^{-M} d\mu(x)d\mu(y)\\
    &=I+II.
\end{align*}
Since $\mu$ is a Frostman measure with exponent $a$, 
\[I\lesssim 2^{j(d-a)}\]
and
\[II\lesssim 2^{jd}\sum_{l=0}^{\infty} 2^{a(l-j)}2^{-lM}\lesssim 2^{j(d-a)}.\]



Thus
$$
\|\mu^j\|_2^2 \lesssim 2^{j(d-a)}.
$$

Taking square root gives the result.
\end{proof}

We now define the incidence measure $\nu$ by
$$
\int f(y)\, d\nu(y) = \int Rf(x)\, d\mu(x)
$$
for all $f\in C_c(\mathbb R^d)$, where $R$ is the Radon transform from Definition \ref{defn: Radon Trans}. 

Using the definition of $R$, this can be written as
$$
\int f(y)\, d\nu(y)
=
\iint_{\phi(x,y)=1} f(y)\psi(x,y)\, d\sigma_x(y)\, d\mu(x).
$$

Since $\psi \ge 0$ and $d\sigma_x$ is a positive measure, we have $Rf(x)\ge 0$ whenever $f\ge 0$. We now justify that the functional
$$
f \mapsto \int Rf(x)\,d\mu(x)
$$
defines a Radon measure.

Since $\psi$ is compactly supported and smooth, and the hypersurfaces $\{y:\phi(x,y)=1\}$ vary smoothly in $x$, there exists a constant $C>0$ such that
$$
\int_{\phi(x,y)=1} \psi(x,y)\,d\sigma_x(y) \le C
$$
uniformly in $x$.

It follows that for every $f \in C_c(\mathbb R^d)$,
$$
|Rf(x)| \le C \|f\|_\infty.
$$
Therefore
$$
\left| \int Rf(x)\,d\mu(x) \right|
\le
C \|f\|_\infty \mu(E),
$$
so the functional is bounded on $C_c(\mathbb R^d)$.

Since it is also linear and positive, the Riesz representation theorem implies that there exists a Radon measure $\nu$ such that
$$
\int f(y)\,d\nu(y)=\int Rf(x)\,d\mu(x)
$$
for all $f \in C_c(\mathbb R^d)$. 

Moreover, $\nu$ is supported on
$$
\bigcup_{x \in E} \{y : \phi(x,y)=1\}.
$$
Indeed, if $f$ vanishes on this set, then $Rf(x)=0$ for all $x$, and hence $\int f\,d\nu=0$. 

We now mollify $\nu$. Let $\rho \in C^\infty_0(\mathbb R^d)$ be non-negative, radial, supported in the unit ball, and satisfy $\|\rho\|_1=1$.

Define
$$
\rho_\epsilon(x) = \epsilon^{-d}\rho(x/\epsilon),
$$
and
$$
\lambda_\epsilon = \nu * \rho_\epsilon.
$$

Then $\lambda_\epsilon$ is smooth and
$$
\|\lambda_\epsilon\|_2^2
=
\int (\nu * \rho_\epsilon)(y)(\nu * \rho_\epsilon)(y)\, dy.
$$

Since $\rho$ is radial, we have $\rho_\epsilon(x)=\rho_\epsilon(-x)$, and therefore
$$
\|\lambda_\epsilon\|_2^2
=
\int (\nu * \rho_\epsilon * \rho_\epsilon)(y)\, d\nu(y).
$$

Define
$$
g_\epsilon = \nu * \rho_\epsilon * \rho_\epsilon.
$$

Then
$$
\|\lambda_\epsilon\|_2^2
=
\int g_\epsilon(y)\, d\nu(y)
=
\int Rg_\epsilon(x)\, d\mu(x).
$$

Under the rotational curvature condition, the operator $R$ is a Fourier integral operator of order $-\frac{d-1}{2}$, and therefore satisfies the estimate
$$
\|R P_j f\|_2 \lesssim 2^{-j\frac{d-1}{2}} \|P_j f\|_2.
$$
which follows from the Fourier integral operator theory of Phong and Stein \cite{PS,PhongStein}.

We now decompose $g_\epsilon$ and $\mu$ into Littlewood--Paley pieces.

Write
$$
g_\epsilon = \sum_{j=0}^{\infty} g_\epsilon^j, \quad g_\epsilon^j = P_j g_\epsilon,
$$
and
$$
\mu = \sum_{k=0}^{\infty} \mu^k.
$$

Then
$$
\|\lambda_\epsilon\|_2^2
=
\sum_{j,k \ge 0} \int Rg_\epsilon^j(x)\, d\mu^k(x).
$$

We split the sum into two parts:
$$
\sum_{|j-k|\le K} \int Rg_\epsilon^j\, d\mu^k
\quad \text{and} \quad
\sum_{|j-k|>K} \int Rg_\epsilon^j\, d\mu^k.
$$

We first estimate the contribution of the comparable frequencies $|j-k|\le K$.

By Cauchy--Schwarz,
$$ \left| \int Rg_\epsilon^j(x)\, d\mu^k(x) \right|
\le
\|R g_\epsilon^j\|_2 \|\mu^k\|_2.
$$

Using the Fourier integral operator estimate,
$$
\|R P_j f\|_2 \lesssim 2^{-j\frac{d-1}{2}} \|P_j f\|_2.
$$

Since $g_\epsilon^j = P_j g_\epsilon$, we obtain
$$
\|R g_\epsilon^j\|_2 = \|R P_j g_\epsilon\|_2 \lesssim 2^{-j\frac{d-1}{2}} \|g_\epsilon^j\|_2.
$$

Using the Frostman estimate,
$$
\|\mu^k\|_2 \lesssim 2^{k\frac{d-a}{2}}.
$$

Thus
$$
\left| \int Rg_\epsilon^j\, d\mu^k \right|
\lesssim 2^{-j\frac{d-1}{2}} 2^{k\frac{d-a}{2}} \|g_\epsilon^j\|_2.
$$

If $|j-k|\le K$, then $k$ may be replaced by $j$ up to constants, and so
$$
\left| \int Rg_\epsilon^j\, d\mu^k \right|
\lesssim 2^{j\frac{1-a}{2}} \|g_\epsilon^j\|_2.
$$

Summing over $k$ with $|j-k|\le K$ gives
$$
\sum_{|j-k|\le K}
\left| \int Rg_\epsilon^j\, d\mu^k \right|
\lesssim
 2^{j\frac{1-a}{2}} \|g_\epsilon^j\|_2.
$$

Summing over $j$,
$$
\sum_{|j-k|\le K}
\left| \int Rg_\epsilon^j\, d\mu^k \right|
\lesssim
 \sum_{j=0}^{\infty} 2^{j\frac{1-a}{2}} \|g_\epsilon^j\|_2.
$$

Since $a>1$, we have $2^{j\frac{1-a}{2}} \in \ell^2$, and therefore by Cauchy--Schwarz,
$$
\sum_{j=0}^{\infty} 2^{j\frac{1-a}{2}} \|g_\epsilon^j\|_2
\le
\left(\sum_{j=0}^{\infty} 2^{j(1-a)}\right)^{1/2}
\left(\sum_{j=0}^{\infty} \|g_\epsilon^j\|_2^2\right)^{1/2}
\lesssim
 \|g_\epsilon\|_2.
$$

We now estimate the contribution of the off-diagonal terms $|j-k|>K$. We use the almost-orthogonality property for Fourier integral operators whose canonical relation is locally a canonical graph. Under the rotational curvature hypothesis, the canonical relation associated to $R$ is nondegenerate and locally a canonical graph; see \cite{Sogge,Stein}. It follows from the standard almost-orthogonality principle for Fourier integral operators whose canonical relation is locally a canonical graph (see, for example, \cite[Chapter VIII]{Sogge}) that frequency localized pieces with widely separated frequencies have rapidly decaying interaction. Therefore, for every $N>0$, there exists $C_N>0$ such that whenever $|j-k|>K$,
$$
\left| \int Rg_\epsilon^j(x)\, d\mu^k(x) \right|
\le
C_N 2^{-N\max(j,k)} \|g_\epsilon^j\|_2 \|\mu^k\|_2.
$$

Using the Frostman estimate,
$$
\|\mu^k\|_2 \lesssim 2^{k\frac{d-a}{2}},
$$
we obtain
$$
\left| \int Rg_\epsilon^j\, d\mu^k \right|
\lesssim
 2^{-N\max(j,k)} 2^{k\frac{d-a}{2}} \|g_\epsilon^j\|_2.
$$

Choose $N$ sufficiently large so that
$$
N > \frac{d-a}{2} + 2.
$$

We consider two cases.

If $k \le j$, then
$$
2^{-N\max(j,k)} 2^{k\frac{d-a}{2}}
=
2^{-Nj} 2^{k\frac{d-a}{2}}
\le
2^{-Nj} 2^{j\frac{d-a}{2}}
=
2^{-j\left(N-\frac{d-a}{2}\right)}
\le
 2^{-2j}.
$$

If $k > j$, then
$$
2^{-N\max(j,k)} 2^{k\frac{d-a}{2}}
=
2^{-Nk} 2^{k\frac{d-a}{2}}
=
2^{-k\left(N-\frac{d-a}{2}\right)}
\le
 2^{-2k}.
$$

In either case,
$$
2^{-N\max(j,k)} 2^{k\frac{d-a}{2}}
\le
 2^{-2\max(j,k)}.
$$

Thus
$$
\left| \int Rg_\epsilon^j\, d\mu^k \right|
\lesssim
 2^{-2\max(j,k)} \|g_\epsilon^j\|_2.
$$

We now sum over $j$ and $k$.

Fix $j$. Split the sum over $k$ into two parts.

If $k \le j$, then
$$
\sum_{k \le j} 2^{-2\max(j,k)} = \sum_{k \le j} 2^{-2j} \lesssim  j2^{-2j}\le 2^{-j}.
$$

If $k > j$, then
$$
\sum_{k > j} 2^{-2\max(j,k)} = \sum_{k > j} 2^{-2k} \lesssim 2^{-j}.
$$

Therefore
$$
\sum_{k=0}^{\infty} 2^{-2\max(j,k)} \lesssim 2^{-j}.
$$

Summing in $j$,
$$
\sum_{|j-k|>K} \left| \int Rg_\epsilon^j\, d\mu^k \right|
\lesssim
 \sum_{j=0}^{\infty} 2^{-j} \|g_\epsilon^j\|_2.
$$

By Cauchy--Schwarz,
$$
\sum_{j=0}^{\infty} 2^{-j} \|g_\epsilon^j\|_2
\le
\left(\sum_{j=0}^{\infty} 2^{-2j}\right)^{1/2}
\left(\sum_{j=0}^{\infty} \|g_\epsilon^j\|_2^2\right)^{1/2}
\lesssim
 \|g_\epsilon\|_2.
$$

We obtain
$$
\sum_{j,k} \left| \int Rg_\epsilon^j\, d\mu^k \right|
\lesssim
 \|g_\epsilon\|_2.
$$

Since $g_\epsilon = \lambda_\epsilon * \rho_\epsilon$ and $\|\rho_\epsilon\|_1=1$ with bounded norm, Young's inequality gives
$$
\|g_\epsilon\|_2 \le  \|\lambda_\epsilon\|_2.
$$

We obtain
$$
\|\lambda_\epsilon\|_2^2 \lesssim \|\lambda_\epsilon\|_2.
$$

Thus $\|\lambda_\epsilon\|_2$ is bounded uniformly in $\epsilon$.

Geometrically, the uniform $L^2$ bound prevents the incidence measure from concentrating on a set of very small Lebesgue measure. Indeed, if most of the mass of $\nu$ were confined to a tiny region, then the mollified densities $\lambda_\epsilon$ would develop large peaks as $\epsilon\to0$, causing the $L^2$ norm to blow up. The boundedness of $\|\lambda_\epsilon\|_2$ therefore forces the incidence measure to spread across a set of positive Lebesgue measure. This is the analytic expression of the underlying geometric phenomenon that curved hypersurfaces cannot overlap too efficiently.

By weak compactness in $L^2$, there exists a subsequence $\lambda_{\epsilon_k}$ and a function $h \in L^2(\mathbb R^d)$ such that $\lambda_{\epsilon_k}$ converges weakly in $L^2$ to $h$. On the other hand, since $\lambda_\epsilon = \nu * \rho_\epsilon$, we have
$$
\int f(y)\lambda_\epsilon(y)\,dy \to \int f(y)\,d\nu(y)
$$
for every continuous compactly supported function $f$. Passing to the same subsequence $\epsilon_k$, we therefore obtain
$$
\int f(y)\lambda_{\epsilon_k}(y)\,dy \to \int f(y)\,d\nu(y),
$$
and also
$$
\int f(y)\lambda_{\epsilon_k}(y)\,dy \to \int f(y)h(y)\,dy.
$$
Thus
$$
\int f(y)h(y)\,dy = \int f(y)\,d\nu(y)
$$
for all $f \in C_c(\mathbb R^d)$, and hence $d\nu = h(y)\,dy$.

Since $\nu$ is a nonzero measure, there exists a non-negative $f \in C_c(\mathbb R^d)$ such that
$$
\int f(y)\,d\nu(y)>0.
$$
It follows that
$$
\int f(y)h(y)\,dy>0,
$$
and hence $h$ is not identically zero. Therefore the set $\{y : h(y)\neq 0\}$ has positive Lebesgue measure, and this set is contained in the support of $\nu$.
\end{proof}

\begin{corollary} 
Under the assumptions of Theorem \ref{thm:fixed level}, let $\nu$ be the incidence measure, and let $h\in L^2(\mathbb R^d)$ denote its density. Then
$$
\nu(\mathbb R^d)^2
\le
\left|
\bigcup_{x \in E}\{y:\phi(x,y)=1\}
\right|
\cdot
\|h\|_2^2.
$$
Consequently,
$$
\left|
\bigcup_{x \in E}\{y:\phi(x,y)=1\}
\right|
\ge
\frac{\nu(\mathbb R^d)^2}{\|h\|_2^2}.
$$
\end{corollary}

\begin{proof}
Since $\nu=h(y)\,dy$ and $h$ is supported on
$$
\bigcup_{x \in E}\{y:\phi(x,y)=1\},
$$
Cauchy--Schwarz gives
$$
\nu(\mathbb R^d)
=
\int h(y)\,dy
\le
\left|
\bigcup_{x \in E}\{y:\phi(x,y)=1\}
\right|^{1/2}
\|h\|_2.
$$
Squaring both sides gives the result.
\end{proof}


\subsection{Vanishing Curvature Case} We now address the case when the rotational curvature is allowed to vanish of finite order. 

\begin{corollary}\label{cor vanishing}
Suppose that the Monge--Amp\`ere determinant associated to $\phi$ does not vanish to infinite order on $\{\phi=1\}$. Let
$$
w(x,y)=
\det
\begin{pmatrix}
0 & \nabla_x \phi \\
-(\nabla_y \phi)^T & \partial^2_{xy}\phi
\end{pmatrix}.
$$
Then there exists $A>0$, depending on the finite type of the degeneracy of $w$, such that the weighted averaging operator
\begin{equation}\label{avg op}
R_{A} f(x)=\int_{\phi(x,y)=1} f(y)\,|w(x,y)|^A \psi(x,y)\,d\sigma_x(y)
\end{equation}
satisfies the $L^2$ Sobolev estimate of order $\frac{d-1}{2}$; see \cite{SoggeStein2}. 

If $E\subset \mathbb R^d$ is a compact set with $\dim_{\mathcal H}(E)>1,$
then
$$
\left|\bigcup_{x\in E}\{y:\phi(x,y)=1\}\right|>0.
$$
\end{corollary}

\begin{remark}  The Korányi Spheres associated to the Heisenberg groups present a particularly interesting example of surfaces that satisfy the hypothesis of the main results in this section.  In particular, while these surfaces have vanishing curvatures, the generalized Radon transform associated to averaging over the Korányi sphere has non-vanishing rotational curvature, in the sense that the associated Monge-Ampere determinant is non-zero (see \cite{Srivastava-Taylor26} for definition and in-depth discussion on the curvature of this object).

\end{remark}

\begin{remark}
The finite-type hypothesis on the Monge--Amp\`ere determinant is needed not only for the proof, but also for the conclusion.

Already in the translation-invariant case this can fail. Let
$$
\phi(x,y)=\phi_0(x-y),
$$
where
$$
\phi_0(z)=\max\{|z_1|,|z_2|\}.
$$
Then
$$
\{z:\phi_0(z)=1\}
$$
is the boundary of the square $[-1,1]^2$. Hence
$$
\{y:\phi(x,y)=1\}=x+\partial[-1,1]^2.
$$
The level set has flat sides, and along those sides, the relevant curvature is absent. In particular, the associated Monge--Amp\`ere determinant vanishes on the flat pieces, so the smoothing mechanism used above is not available.

Let
$$
E=C\times C,
$$
where $C\subset [0,1]$ is the middle-third Cantor set. Then
$$
\dim_{\mathcal H}(E)=2\log_3 2>1.
$$

With
$$
\phi(x,y)=\max\{|x_1-y_1|,|x_2-y_2|\},
$$
we have 
$$
\{y:\phi(x,y)=1\}=x+\partial[-1,1]^2,
$$
and it follows that 
$$
\bigcup_{x\in E}\{y:\phi(x,y)=1\}
=
E+\partial[-1,1]^2.
$$

Since
$$
E+\partial[-1,1]^2
\subset
(C+[-1,1])\times (C+\{-1,1\})
\cup
(C+\{-1,1\})\times (C+[-1,1]),
$$
and $C+\{-1,1\}$ has one-dimensional Lebesgue measure zero, the union above has two-dimensional Lebesgue measure zero.

Thus, when the level set $\{\phi_0=1\}$ contains flat pieces, even a positive-measure family of translates may fail to produce a set of positive Lebesgue measure. This shows that the finite-type nondegeneracy assumption is a genuine geometric hypothesis, not just a technical condition imposed by the proof.
\end{remark} 

\begin{figure}[t]
\centering
\begin{tikzpicture}[scale=0.8]

\begin{scope}
\node at (0,2) {curved case};

\foreach \x in {-1,0,1}
{
\fill (\x,0) circle (0.06);
\draw[gray] (\x,0) circle (1);
}

\node at (0,-2) {overlap fills area};
\end{scope}

\begin{scope}[xshift=8cm]
\node at (0,2) {flat case};

\foreach \x in {-1,0,1}
{
\fill (\x,0) circle (0.06);

\draw[gray] (\x-1,-1) -- (\x+1,-1);
\draw[gray] (\x-1,1) -- (\x+1,1);
\draw[gray] (\x-1,-1) -- (\x-1,1);
\draw[gray] (\x+1,-1) -- (\x+1,1);
}

\node at (0,-2) {stacked without filling};
\end{scope}

\end{tikzpicture}
\caption{Curvature versus flat geometry. Curved hypersurfaces spread in transverse directions and fill area, while flat pieces can stack without producing positive measure.}
\end{figure}
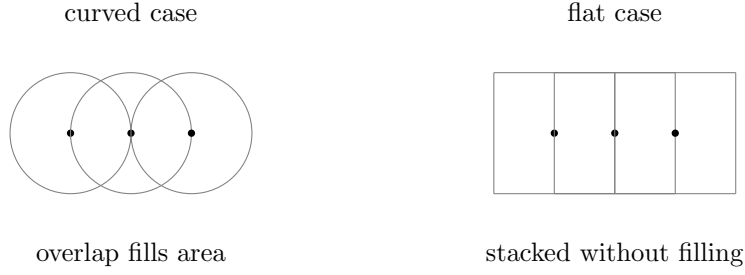

\begin{proof}
The proof of Corollary \ref{cor vanishing} follows the same argument as that of Theorem \ref{thm:fixed level}, with $R$ replaced by $R_A$ defined in \eqref{avg op}. By the theorem of Sogge and Stein on weighted averages over hypersurfaces \cite{SoggeStein2}, the finite type assumption on the Monge--Amp\`ere determinant implies that there exists $A>0$ such that $R_A$ satisfies the same $L^2$ Sobolev estimate as in the nondegenerate case.

Let $\mu$ be a Frostman measure supported on $E$ with exponent $s>1$. Define the incidence measure $\nu$ by
$$
\int f(y)\,d\nu(y)=\int R_A f(x)\,d\mu(x).
$$
Since $|w(x,y)|^A\psi(x,y)\ge 0$, the measure $\nu$ is nonnegative. We now verify that $\nu$ has positive total mass. Since $w$ does not vanish to infinite order on the relevant support, there exists a coordinate patch on which
$$
\int_{\phi(x,y)=1} |w(x,y)|^A \psi(x,y)\,d\sigma_x(y)
$$
is strictly positive for $x$ in an open set. By localizing $\psi$ to this patch, we may assume that this quantity is positive for $x$ in a set of positive $\mu$-measure. It follows that
$$
\nu(\mathbb R^d)
=
\int \int_{\phi(x,y)=1} |w(x,y)|^A \psi(x,y)\,d\sigma_x(y)\,d\mu(x)
>0.
$$

The Littlewood--Paley decomposition and the $L^2$ estimates for $R_A$ yield the same bounds as in the proof of the main theorem, and show that a mollification of $\nu$ is uniformly bounded in $L^2$. It follows that $\nu$ is absolutely continuous with an $L^2$ density.

Since $\nu$ is nonzero, its support has positive Lebesgue measure. This support is contained in
$$
\bigcup_{x\in E}\{y:\phi(x,y)=1\},
$$
which therefore also has positive Lebesgue measure.
\end{proof}

\subsection{A Discrete Incidence Theoretic Consequence} The fixed-level theorem has a natural finite-scale interpretation connected to geometric incidence theory. Roughly speaking, Sobolev smoothing forces sufficiently distributed families of hypersurfaces to occupy a large portion of the ambient space. The following corollary illustrates this principle in a model discrete setting.

\begin{corollary}[A finite-scale separated set consequence]
Let $1<s<2$. For each large integer $q$, let $P_q\subset [0,q]^2$ contain one point in each unit lattice square, and assume that the points of $P_q$ are uniformly separated, independently of $q$. Set
$$
\rho_q=q^{-\frac{2}{s}},
$$
and define
$$
E_q=\left(q^{-1}P_q\right)_{\rho_q}\subset \mathbb R^2.
$$
Fix $r>0$. Then there exists a constant $c>0$, independent of $q$, such that
$$
\left|
\bigcup_{x\in E_q}\{y\in\mathbb R^2: |x-y|=r\}
\right|
\ge c.
$$
Equivalently, after rescaling,
$$
\left|
\bigcup_{p\in P_q}
\left\{z\in\mathbb R^2:\big||z-p|-rq\big|\le q\rho_q\right\}
\right|
\ge c q^2.
$$
Since $q\rho_q=q^{1-\frac{2}{s}}$, this says that the union of annuli
$$
\left\{z:\big||z-p|-rq\big|\le q^{1-\frac{2}{s}}\right\},
\qquad p\in P_q,
$$
has area bounded below by a fixed positive proportion of the ambient scale $q^2$.
\end{corollary}

\begin{proof}
Let $\mu_q$ be the normalized Lebesgue measure on $E_q$. Since $P_q$ contains one point in each unit lattice square and the points are uniformly separated, the set $q^{-1}P_q$ consists of approximately $q^2$ points separated by $\gtrsim q^{-1}$.

We first record the uniform Frostman estimate. Since $\rho_q=q^{-2/s}$, we have
$$
\rho_q^s=q^{-2}.
$$
For balls of radius $t\le \rho_q$, the normalized Lebesgue measure of the intersection with one $\rho_q$-ball is bounded by
$$
C q^{-2}\left(\frac{t}{\rho_q}\right)^2
\le C t^s.
$$
For $\rho_q\le t\le q^{-1}$, a ball of radius $t$ meets only a bounded number of the components of $E_q$, and hence
$$
\mu_q(B(x,t))\le C q^{-2}=C\rho_q^s\le C t^s.
$$
For $t\ge q^{-1}$, the one point per unit square condition implies that a ball of radius $t$ meets at most $C(qt)^2$ of the components. Therefore
$$
\mu_q(B(x,t))\le C(qt)^2 q^{-2}=Ct^2\le Ct^s,
$$
since $t\le 1$ and $s<2$. Thus
$$
\mu_q(B(x,t))\le C t^s
$$
uniformly in $q$.

Now apply the proof of the main theorem with this uniformly Frostman family of measures. Since $s>1$, the Littlewood--Paley summation in the proof is uniformly bounded. Therefore the incidence measure associated to
$$
\bigcup_{x\in E_q}\{y:|x-y|=r\}
$$
has an $L^2$ density with norm bounded independently of $q$, while its total mass is bounded below independently of $q$. It follows by Cauchy--Schwarz that the support of this incidence measure has Lebesgue measure bounded below independently of $q$. This gives
$$
\left|
\bigcup_{x\in E_q}\{y:|x-y|=r\}
\right|
\ge c.
$$

Finally, rescale by the map $z=qy$. Lebesgue measure is multiplied by $q^2$, and the circle of radius $r$ centered at $x=q^{-1}p+O(\rho_q)$ becomes an annulus of radius $rq$ and thickness comparable to $q\rho_q=q^{1-2/s}$ centered at $p$. This gives the stated rescaled form.
\end{proof}

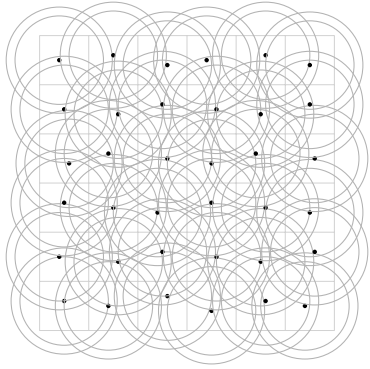
\begin{figure}[h]
\centering
\begin{tikzpicture}[scale=0.65]

\draw[step=1, gray!35, very thin] (0,0) grid (6,6);

\foreach \x/\y in {
0.5/0.6, 1.4/0.5, 2.6/0.7, 3.5/0.4, 4.6/0.6, 5.4/0.5,
0.4/1.5, 1.6/1.4, 2.5/1.6, 3.6/1.5, 4.5/1.4, 5.6/1.6,
0.5/2.6, 1.5/2.5, 2.4/2.4, 3.5/2.6, 4.6/2.5, 5.5/2.4,
0.6/3.4, 1.4/3.6, 2.6/3.5, 3.5/3.4, 4.4/3.6, 5.6/3.5,
0.5/4.5, 1.6/4.4, 2.5/4.6, 3.6/4.5, 4.5/4.4, 5.5/4.6,
0.4/5.5, 1.5/5.6, 2.6/5.4, 3.4/5.5, 4.6/5.6, 5.5/5.4}
{
\fill (\x,\y) circle (0.045);
\draw[gray!60] (\x,\y) circle (0.9);
\draw[gray!60] (\x,\y) circle (1.08);
}

\node at (3,-1.3) {uniformly distributed centers};
\node at (3,7.4) {overlapping annuli at a fixed scale};

\end{tikzpicture}
\caption{A schematic illustration of the finite-scale consequence: a uniformly distributed set of points gives rise to a family of annuli whose union occupies a large region.}
\end{figure}

\begin{remark}[Incidence interpretation]
The preceding corollary admits a natural interpretation in the language of incidence geometry. Let $P_q\subset [0,q]^2$ be as above, with one point in each unit lattice square and with uniform separation. Consider the family of annuli
$$
\mathcal A_p=\left\{z\in\mathbb R^2:\big||z-p|-rq\big|\le q^{1-\frac{2}{s}}\right\},
\qquad p\in P_q.
$$
The corollary shows that
$$
\left|\bigcup_{p\in P_q}\mathcal A_p\right|\gtrsim q^2,
$$
with constants independent of $q$.

Equivalently, if one defines the incidence count
$$
I(z)=\#\{p\in P_q: z\in \mathcal A_p\},
$$
then
$$
\int I(z)\,dz \gtrsim q^2,
$$
which shows that the total incidence mass is of order $q^2$, and in particular the union $\bigcup_{p\in P_q}\mathcal A_p$ occupies a set of area $\gtrsim q^2$. In this sense, the family $\{\mathcal A_p\}_{p\in P_q}$ produces a large-scale geometric coverage, reflecting a high total incidence mass in the continuum setting.

This perspective is closely related to the conversion mechanism developed in \cite{IosevichJoratiLaba}, where Sobolev bounds for generalized Radon transforms are used to derive upper bounds on discrete incidences between homogeneous point sets and families of hypersurfaces. In the present setting, the argument operates in a complementary direction: instead of producing upper bounds on the number of incidences, the same smoothing mechanism forces large-scale geometric coverage by the associated family of hypersurfaces. In this sense, the finite-scale conclusion above may be viewed as a continuous analogue of the discrete incidence bounds obtained in \cite{IosevichJoratiLaba}. The present argument shows that the same Sobolev smoothing estimates that yield upper bounds for discrete incidence counts in \cite{IosevichJoratiLaba} also enforce a form of large-scale coverage in the continuum model considered here.

\end{remark}




\section{The Variable Level Result}
\subsection{The Maximal Operator Argument}

We now consider the case where the level is allowed to vary with the parameter. 
Recall the maximal operator
$$
\mathcal{M}R_I f(x)=\sup_{t \in I} |R_t f(x)|,
$$
where $I \subset (0,\infty)$ is a fixed compact interval.

Under the rotational curvature hypothesis, the operators $R_t$ form a smoothly parametrized family of Fourier integral operators of order $-\frac{d-1}{2}$. Standard results on maximal averages over curved hypersurfaces imply that $\mathcal{M}R_I$ satisfies a corresponding $L^2$ Sobolev estimate with gain $\frac{d-2}{2}-\epsilon$ derivatives for every $\epsilon>0$; see \cite{SoggeStein2,Stein}.

This leads to the following maximal variant of the main theorem. The dimension threshold obtained below should be viewed as reflecting the limitations of the maximal-operator method rather than the expected geometric threshold for all variable-parameter families. As the later intersection theorem shows, additional geometric control can restore the critical threshold $\dim_{\mathcal H}(E)>1$.

\begin{theorem}
Let $I\subset (0,\infty)$ be a fixed compact interval. Suppose $\phi\in C^\infty(\mathbb R^d\times\mathbb R^d),d\geq 3$ satisfies the rotational curvature condition on $I$, and let $t:\mathbb R^d \to I$ be an arbitrary function. If $E\subset\mathbb R^d$ is a compact set with $\dim_{\mathcal H}(E)>2$, then
$$
\left| \bigcup_{x \in E} \{y : \phi(x,y)=t(x)\} \right|>0.
$$
\end{theorem}

\begin{proof}
The proof follows the same argument as in the fixed-level case. For each $x\in E$,
$$
R_{t(x)}f(x) \le \mathcal{M}R_I f(x),
$$
and therefore the incidence measure associated to the variable level set, defined by
$$
\int f(y)\,d\nu(y)=\int R_{t(x)}f(x)\,d\mu(x),
$$
is controlled by the maximal operator $\mathcal{M}R_I$.

We now justify the maximal estimate. For a smooth function $F$ on $I$, the one-dimensional Sobolev inequality gives
$$
\sup_{t\in I}|F(t)|^2
\lesssim
\left(\int_I |F(t)|^2\,dt\right)^{1/2}
\left(\int_I |F'(t)|^2\,dt\right)^{1/2}.
$$

Apply this pointwise in $x$ with
$$
F(t)=R_tP_jf(x).
$$

Integrating in $x$ and applying Cauchy--Schwarz in $x$, we obtain 
$$
\|\mathcal{M}R_IP_jf\|_2^2
\lesssim
\left(\int_I \|R_tP_jf\|_2^2\,dt\right)^{1/2}
\left(\int_I \|\partial_tR_tP_jf\|_2^2\,dt\right)^{1/2}.
$$

Taking square roots, we obtain
$$
\|\mathcal{M}R_IP_j f\|_2
\lesssim
\left(\int_I \|R_tP_jf\|_2^2\,dt\right)^{1/4}
\left(\int_I \|\partial_tR_tP_jf\|_2^2\,dt\right)^{1/4}.
$$

The fixed-time estimate gives
$$
\|R_tP_jf\|_2\lesssim 2^{-j\frac{d-1}{2}}\|P_jf\|_2,
$$
uniformly in $t\in I$. Differentiating in $t$ costs one derivative at frequency $2^j$, so
$$
\|\partial_tR_tP_jf\|_2
\lesssim
2^j2^{-j\frac{d-1}{2}}\|P_jf\|_2.
$$
It follows that
$$
\|\mathcal{M}R_IP_jf\|_2
\lesssim
2^{-j\frac{d-2}{2}}\|P_jf\|_2.
$$
Allowing an arbitrarily small loss, we shall use the estimate
$$
\|\mathcal{M}R_I P_j f\|_2 \le C_\epsilon 2^{-j\left(\frac{d-2}{2}-\epsilon\right)} \|P_j f\|_2.
$$

Proceeding as in the fixed-level case and using a Frostman measure with exponent $a>2$, we obtain 
$$
2^{-j\left(\frac{d-2}{2}-\epsilon\right)} 2^{j\frac{d-a}{2}}
=
2^{j\left(\frac{2-a}{2}+\epsilon\right)}.
$$
Choosing $\epsilon>0$ sufficiently small, the sum is finite whenever $a>2$. The remainder of the argument proceeds as in the proof of the main theorem, with $R$ replaced by $\mathcal{M}R_I$, and shows that the associated incidence measure is absolutely continuous with an $L^2$ density. This implies the desired positivity conclusion.
\end{proof}

\begin{remark}[Comparison with Mitsis]
The results of Mitsis \cite{Mitsis} for spheres provide the model case for the geometric mechanism isolated below. In the planar case, Mitsis proves that if $E \subset \mathbb R^2$ satisfies $\dim_{\mathcal H}(E)>1$, then
$$
\left|\bigcup_{x\in E}\{y:|y-x|=r(x)\}\right|>0
$$
for an arbitrary measurable radius function $r(x)$ taking values in a fixed compact interval.

We briefly record the standard argument controlling the maximal operator in terms of square functions in the parameter $t$. Fix $x$ and define
$$
F(t)=R_tP_j f(x).
$$
For any $t\in I$, choose $t_0 \in I$ and write
$$
F(t)=F(t_0)+\int_{t_0}^t F'(s)\,ds.
$$
Using $(a+b)^2 \le 2a^2+2b^2$, we obtain
$$
|F(t)|^2
\le
2|F(t_0)|^2
+
2\left|\int_{t_0}^t F'(s)\,ds\right|^2.
$$
By Cauchy--Schwarz,
$$
\left|\int_{t_0}^t F'(s)\,ds\right|^2
\le
|I|\int_I |F'(s)|^2 ds.
$$
Averaging the first term over $t_0 \in I$ gives
$$
|F(t_0)|^2 \le \frac{1}{|I|}\int_I |F(s)|^2 ds,
$$
and therefore
$$
|F(t)|^2
\le
C\int_I |F(s)|^2 ds
+
C\int_I |F'(s)|^2 ds.
$$
Taking the supremum over $t\in I$, we obtain
$$
\sup_{t\in I} |F(t)|^2
\le
C\int_I |F(s)|^2 ds
+
C\int_I |F'(s)|^2 ds.
$$
Returning to $F(t)=R_tP_j f(x)$ and integrating in $x$, we conclude
$$
\|\mathcal{M}R_IP_j f\|_2^2
\le
C\int_I \|R_sP_j f\|_2^2 ds
+
C\int_I \|\partial_s R_sP_j f\|_2^2 ds.
$$

The preceding estimate, combined with the known $L^2$ Sobolev bounds for $R_t$, yields an $L^2$ Sobolev improvement of order $\frac{d-2}{2}-\epsilon$ for every $\epsilon>0$, which leads to the condition $\dim_{\mathcal H}(E)>2$ in the variable-level setting.
\end{remark}

\subsubsection{A Counterexample for Variable Directions}

Let $ E$ denote a subset of $\R^d$ satisfying $\dim_{\rm H}(E) >1$. While it is a consequence of Theorem \ref{thm:fixed level} that $\bigcup_{x\in E} \,\{y: x\cdot y =1\}$ has positive measure, the next example demonstrates that this need not hold in the variable coefficient setting. 

\begin{lemma}[Kakeya Construction]{}
There exists a compact set $E \subset \mathbb R^2$ with positive Lebesgue measure such that for each $x \in E$ one can choose a line segment contained in a line of the form
$$
\{y \in \mathbb R^2 : x \cdot y = t(x)\}
$$
whose union has Lebesgue measure zero.
\end{lemma}

\begin{proof}
It is known that there exists a compact set $K \subset \mathbb R^2$ of Lebesgue measure zero that contains a unit line segment in every direction; see \cite{Besicovitch}.

Let $E$ be any compact set of positive Lebesgue measure. For each $x \in E$, choose a unit line segment contained in $K$ whose normal vector is parallel to $x$. Since $K$ contains a unit segment in every direction, this is always possible.

Each such segment lies on a line of the form
$$
\{y \in \mathbb R^2 : x \cdot y = t(x)\}
$$
for some real number $t(x)$. 

By construction, all the selected segments lie inside $K$, and therefore their union has Lebesgue measure zero.
\end{proof}

This result shows that, in contrast with the variable radius circle problem studied by Mitsis \cite{Mitsis}, arbitrary measurable parameter dependence can lead to Kakeya-type counterexamples in variable coefficient settings.

\begin{example}[Bourgain's curved Kakeya example]
The Kakeya-type obstruction in the preceding lemma is not merely an analogy. Bourgain constructed a curved Kakeya example in $\mathbb R^3$ showing that a family of curves with all relevant directions represented may nevertheless be confined to a two-dimensional set.

The example arises from the phase
$$
\phi(x,t;y)=x\cdot y+ty_1y_2+\frac12 t^2y_1^2,
$$
which was introduced by Bourgain in his counterexamples for oscillatory integral estimates. The associated characteristic curves are obtained from
$$
\partial_y \phi(x,t;y)=w.
$$
Solving this equation gives curves of the form
$$
\{(w_1-ty_2-t^2y_1,w_2-ty_1,t): |t|\le 1\}.
$$
If one chooses $w_1=0$ and $w_2=-y_2$, the resulting family is
$$
\bigcup_{y_1,y_2}
\{(-ty_2-t^2y_1,-y_2-ty_1,t): |t|\le 1\}.
$$
Writing $(X,Y,Z)=(-ty_2-t^2y_1,-y_2-ty_1,t)$, one checks directly that
$$
X=YZ.
$$
Thus the whole family is contained in the two-dimensional hypersurface
$$
\{(X,Y,Z)\in \mathbb R^3:X=YZ\}.
$$

In higher co-dimension and variable-parameter settings, the example shows that curvature of the individual hypersurfaces alone does not prevent compression phenomena. What matters is not only the geometry of each individual member of the family, but also the way the family is allowed to vary with the parameter. The fixed-parameter results proved earlier avoid this pathology because the hypersurfaces remain tied to a common geometric scale, while arbitrary measurable parameter selections may destroy the transverse spreading responsible for positivity.

\end{example}

\begin{figure}[h]
\centering

\begin{tikzpicture}[scale=0.9]

\draw (-6,-3) rectangle (-0.5,3);

\node at (-3.25,3.6) {fixed level $t=1$};

\draw[thick]
(-5.5,-2)
.. controls (-4,-0.5) and (-2.5,0.5) ..
(-1,2);

\draw[thick]
(-5.5,2)
.. controls (-4,0.5) and (-2.5,-0.5) ..
(-1,-2);

\draw[thick]
(-5.3,-2.4)
.. controls (-3.8,-1) and (-2.3,0) ..
(-0.8,2.3);

\draw[->,thick] (-3.4,-0.3) -- (-4.3,-1.2);
\draw[->,thick] (-3.1,0.1) -- (-2,1.2);

\node at (-3.2,-3.6)
{\small transverse spreading};

\draw (0.5,-3) rectangle (6,3);

\node at (3.25,3.6) {variable measurable $t(x)$};

\draw[thick]
(1,-2)
.. controls (2.5,-1) and (4,0.5) ..
(5.5,2);

\draw[thick]
(1.1,-1.9)
.. controls (2.6,-0.9) and (4.1,0.6) ..
(5.6,2.1);

\draw[thick]
(1.2,-1.8)
.. controls (2.7,-0.8) and (4.2,0.7) ..
(5.7,2.2);

\fill[gray!35]
(2.1,-1.2)
.. controls (3,-0.5) and (4,0.5) ..
(4.9,1.2)
-- (5.2,0.9)
.. controls (4.2,0.2) and (3.2,-0.8) ..
(2.3,-1.5)
-- cycle;

\draw[->,thick] (2.2,1.7) -- (3.3,0.6);
\draw[->,thick] (4.8,-1.7) -- (3.8,-0.5);

\node at (3.25,-3.6)
{\small compression and concentration};

\end{tikzpicture}

\caption{Fixed-level families remain tied to a common geometric scale and exhibit transverse spreading, while arbitrary measurable parameter selections may produce compression phenomena analogous to Kakeya concentration.}
\end{figure}
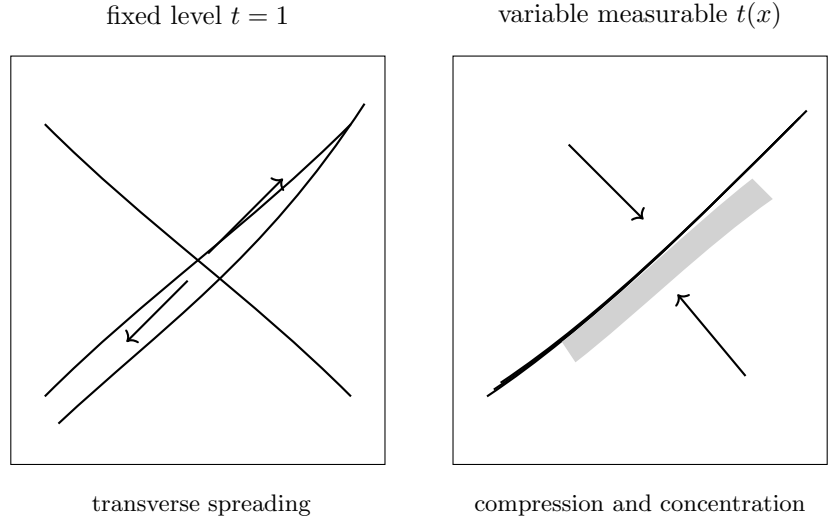

\subsection{The Intersection Hypothesis}

We now state a complementary result showing that the full strength of the variable-parameter Mitsis theorem can be recovered under a direct geometric intersection hypothesis. The point is that the maximal-operator approach treats all measurable selections $t=t(x)$ simultaneously and therefore suffers from Kakeya-type compression phenomena, while the intersection hypothesis imposes a direct geometric transversality condition that prevents such compression. This result is logically independent of the Fourier integral operator approach developed below, and is included to isolate a geometric mechanism that recovers the full strength of Mitsis's variable-parameter theorem. If the family $\Sigma_x=\{y:\phi(x,y)=t(x)\}$ satisfies a uniform Mitsis-type intersection condition, then the positivity conclusion holds for arbitrary measurable selections $t=t(x)$, with the same threshold $\dim_{\mathcal H}(E)>1$.

\begin{theorem}[Variable coefficient Mitsis theorem under an intersection hypothesis]
Let $\phi\in C^\infty(\mathbb R^d\times\mathbb R^d),d\ge 2$ with $\nabla_y\phi(x,y)\ne 0$, and let $t:\mathbb R^d\to I$ be a measurable function, where $I$ is a fixed compact interval. For each $x \in \mathbb R^d$, define
$$
\Sigma_x:=\{y \in \mathbb R^d : \phi(x,y)=t(x)\}.
$$

Suppose that there exists a universal constant $C>0$ such that for all $x,x'\in\mathbb R^d$ and all sufficiently small $\delta>0$,
\begin{equation}\label{eqn:intersection hypo}
|\Sigma_x^\delta \cap \Sigma_{x'}^\delta|
\leq C \frac{\delta^2}{\delta+|x-x'|},
\end{equation}
where $\Sigma_x^\delta$ denotes the $\delta$-neighborhood of $\Sigma_x$. 

If $E \subset \mathbb R^d$ is a compact set with $\dim_{\mathcal H}(E)>1$, then
$$
\left|
\bigcup_{x \in E} \Sigma_x
\right|>0.
$$
\end{theorem}

\begin{remark}
     The intersection hypothesis (\ref{eqn:intersection hypo}) asserts that the $\delta$-neighborhoods of distinct hypersurfaces intersect in a quantitatively controlled way, decaying inversely with the separation of the parameters. It rules out large-scale alignment or concentration phenomena among the family $\{\Sigma_x\}$.
\end{remark}

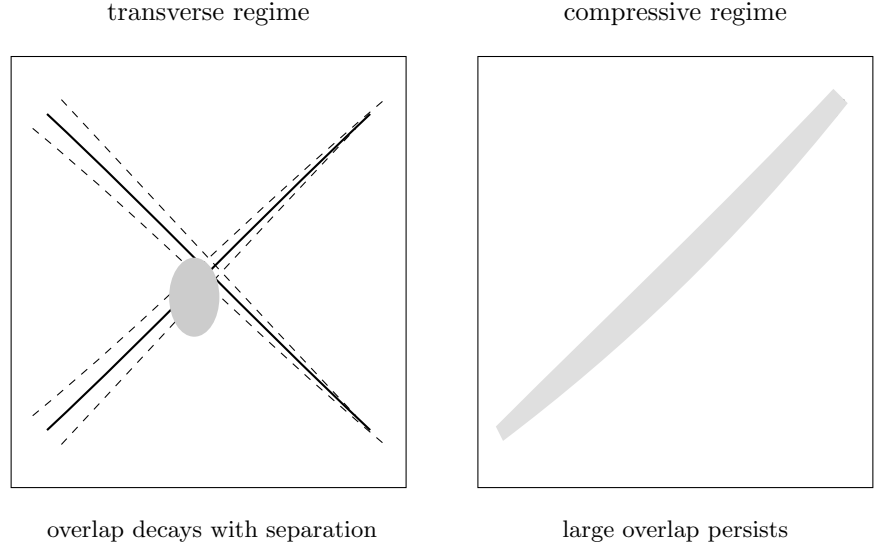
\begin{figure}[h]
\centering

\begin{tikzpicture}[scale=0.95]

\draw (-6,-3) rectangle (-0.5,3);

\node at (-3.25,3.6) {transverse regime};

\draw[thick]
(-5.5,-2.2)
.. controls (-4,-0.8) and (-2.8,0.5) ..
(-1,2.2);

\draw[thick]
(-5.5,2.2)
.. controls (-4,0.8) and (-2.8,-0.5) ..
(-1,-2.2);

\draw[dashed]
(-5.7,-2.0)
.. controls (-4,-0.6) and (-2.8,0.7) ..
(-0.8,2.4);

\draw[dashed]
(-5.3,-2.4)
.. controls (-4,-1.0) and (-2.8,0.3) ..
(-1.2,2.0);

\draw[dashed]
(-5.7,2.0)
.. controls (-4,0.6) and (-2.8,-0.7) ..
(-0.8,-2.4);

\draw[dashed]
(-5.3,2.4)
.. controls (-4,1.0) and (-2.8,-0.3) ..
(-1.2,-2.0);

\fill[gray!40] (-3.45,-0.35) ellipse (0.35 and 0.55);

\node at (-3.2,-3.6)
{\small overlap decays with separation};

\draw (0.5,-3) rectangle (6,3);

\node at (3.25,3.6) {compressive regime};

\draw[thick]
(1,-2.2)
.. controls (2.5,-1) and (4,0.5) ..
(5.5,2.2);

\draw[thick]
(1.1,-2.0)
.. controls (2.6,-0.8) and (4.1,0.7) ..
(5.6,2.4);

\fill[gray!25]
(0.85,-2.35)
.. controls (2.5,-1.1) and (4.1,0.4) ..
(5.65,2.35)
-- (5.45,2.55)
.. controls (3.9,0.9) and (2.3,-0.6) ..
(0.75,-2.15)
-- cycle;

\node at (3.25,-3.6)
{\small large overlap persists};

\end{tikzpicture}

\caption{Schematic illustration of the Mitsis-type intersection condition.
Left: transverse hypersurfaces whose $\delta$-neighborhood intersections shrink as the parameter separation increases.
Right: a compressive configuration where selected hypersurfaces nearly coincide, violating the required overlap bound.
}
\end{figure}

\begin{proof}
Let $\mu$ be a Frostman measure supported on $E$ with exponent $a>1$, so that
$$
\mu(B(x,r))\lesssim r^a.
$$
For $\delta>0$, define
$$
M_\delta f(x)=\frac{1}{|\Sigma_x^\delta|}\int_{\Sigma_x^\delta} f(y)\, dy.
$$
Assume throughout that $\nabla_y \phi(x,y)\neq 0$ on the support of the family under consideration, so that each $\Sigma_x$ is a smooth hypersurface. After localizing to a compact region, standard tubular neighborhood estimates imply that
$$
|\Sigma_x^\delta|\approx \delta
$$
uniformly in $x$ and sufficiently small $\delta$.

We claim that
\begin{equation}
\int_E |M_\delta f(x)|^2 d\mu(x)\lesssim \|f\|_2^2,
\end{equation}
with a implicit constant independent of $\delta$. By duality, it is enough to prove the corresponding adjoint estimate for the operator $f \mapsto M_\delta f$. Let $g\in L^2(\mu)$. Then
$$
\left\|\int_E g(x)\frac{\chi_{\Sigma_x^\delta}}{|\Sigma_x^\delta|} d\mu(x)\right\|_2^2
=
\int_E\int_E g(x)g(x')
\frac{|\Sigma_x^\delta\cap \Sigma_{x'}^\delta|}{|\Sigma_x^\delta||\Sigma_{x'}^\delta|}
d\mu(x)d\mu(x').
$$
Using $|\Sigma_x^\delta|\approx \delta$ and the assumed intersection estimate, the kernel is bounded by
$$ \frac{|\Sigma_x^\delta\cap \Sigma_{x'}^\delta|}{|\Sigma_x^\delta||\Sigma_{x'}^\delta|} 
\lesssim \frac{1}{\delta+|x-x'|}.
$$
It remains to show that
$$
\int_E \frac{d\mu(x')}{\delta+|x-x'|}
$$
is bounded uniformly in $x$ and $\delta$. This follows from the Frostman condition. Indeed,
$$
\int_E \frac{d\mu(x')}{\delta+|x-x'|}
\le
\frac{\mu(B(x,\delta))}{\delta}
+
\sum_{k: 2^k\delta\le 1}
\frac{\mu(B(x,2^{k+1}\delta))}{2^k\delta}.
$$
Using $\mu(B(x,r))\lesssim r^a$, the right hand side is bounded by
$$
\delta^{a-1}
+
\sum_{2^k\delta\le 1}
2^{k(a-1)}\delta^{a-1}
\lesssim 1,
$$
because $a>1$. Schur's lemma gives the claimed $L^2(\mu)$ estimate.

Let
$$
U=\bigcup_{x\in E}\Sigma_x.
$$

Suppose, for contradiction, that
$$
|U|=0.
$$

Let $U^\delta$ denote the $\delta$-neighborhood of $U$. Since $|U|=0$, outer regularity of Lebesgue measure implies that
$$
|U^\delta|\to 0
\quad \text{as } \delta\to0.
$$

Since $\Sigma_x^\delta\subset U^\delta$, we have
$$
M_\delta \chi_{U^\delta}(x)=1
$$
for every $x\in E$. Hence
$$
\mu(E)\le
\int_E |M_\delta\chi_{U^\delta}(x)|^2 d\mu(x)
\lesssim |U^\delta|.
$$
Letting $\delta\to0$ gives $\mu(E)=0$, a contradiction. This contradiction shows that $|U|>0$. Unlike the Fourier integral operator argument used earlier, the proof here is purely geometric and relies directly on quantitative overlap control for the hypersurface family. In particular, the recovery of the threshold $\dim_{\mathcal H}(E)>1$ reflects the fact that the intersection hypothesis rules out the Kakeya-type compression phenomena responsible for the derivative loss in the maximal-operator approach.
\end{proof} 

The intersection hypothesis in the preceding theorem is not automatic, even under strong curvature assumptions on the defining function $\phi$. The following example shows that one may have rotational curvature, and even cinematic curvature, while the induced family $\Sigma_x=\{y:\phi(x,y)=t(x)\}$ fails the required pairwise intersection bound. 

\begin{remark}[The intersection condition is independent of rotational curvature]
The intersection hypothesis in the preceding theorem is not a consequence of standard curvature assumptions on the full family of hypersurfaces, including rotational curvature and the usual cinematic curvature conditions.

Indeed, work in $\mathbb R^d$, $d\ge 3$, and write $x=(x',x_d)$ and $y=(y',y_d)$, with $x',y'\in\mathbb R^{d-1}$. Define
$$
\phi(x,y)=y_d-x_d-|y'-x'|^2.
$$
The level sets
$$
\phi(x,y)=t
$$
are translated paraboloids
$$
y_d=x_d+|y'-x'|^2+t.
$$

Now choose
$$
t(x)=1-x_d.
$$
Then
$$
\phi(x,y)=t(x)
$$
becomes
$$
y_d=1+|y'-x'|^2.
$$
Restrict to the parameter set
$$
\{x=(0,x_d):x_d\in [0,1]\}.
$$
For every such $x$ we have $x'=0$, and hence
$$
\Sigma_x=\{y:y_d=1+|y'|^2\}.
$$
Thus all the selected hypersurfaces are identical.

Consequently, if $x=(0,a)$ and $x'=(0,b)$ with $|a-b|\approx 1$, then
$$
\Sigma_x^\delta=\Sigma_{x'}^\delta,
$$
and on any fixed bounded region,
$$
|\Sigma_x^\delta\cap \Sigma_{x'}^\delta|\approx \delta.
$$
However, the Mitsis-type intersection condition would require
$$
|\Sigma_x^\delta\cap \Sigma_{x'}^\delta|
\le C\frac{\delta^2}{\delta+|x-x'|}
\lesssim \delta^2.
$$
This is false for small $\delta$.

Thus the intersection hypothesis in the theorem is a genuine geometric assumption on the selected graph family $t=t(x)$, and cannot be replaced merely by rotational curvature or cinematic curvature of the full family $\phi(x,y)=t$.
\end{remark}

\begin{remark}[A genuinely variable coefficient family satisfying the intersection condition]
There are genuinely variable coefficient families for which the intersection condition is satisfied. Let $\Phi:\mathbb R^d\to\mathbb R^d$ be a smooth diffeomorphism whose Jacobian and inverse Jacobian are uniformly bounded on the compact region under consideration, and suppose that $\Phi$ is not affine. Define
$$
\phi(x,y)=|\Phi(y)-x|.
$$
Let $t(x)$ be any measurable function satisfying
$$
0<c\le t(x)\le C<\infty,
$$
and set
$$
\Sigma_x=\{y:\phi(x,y)=t(x)\}.
$$
Equivalently,
$$
\Sigma_x=\Phi^{-1}\{z:|z-x|=t(x)\}.
$$

Since $\Phi$ is bi-Lipschitz on the compact region under consideration, the image of $\Sigma_x^\delta$ under $\Phi$ is contained in a $C\delta$-neighborhood of the Euclidean sphere $\{z:|z-x|=t(x)\}$, and conversely contains a $c\delta$-neighborhood of that sphere. Therefore, for $d\ge 3$,
$$
|\Sigma_x^\delta\cap \Sigma_{x'}^\delta|
\le C\frac{\delta^2}{\delta+|\Phi(x)-\Phi(x')|}.
$$
Since $\Phi$ is bi-Lipschitz on the compact region under consideration, $|\Phi(x)-\Phi(x')|\approx |x-x'|$, and the stated intersection bound follows.

Choosing a nonlinear diffeomorphism $\Phi$ gives a genuinely variable coefficient example.
\end{remark}

\begin{remark}[An additional example supporting the intersection hypothesis]
A related example appears in work of Mitsis \cite{MitsisPlane}. In $\mathbb R^3$, let $D \subset S^2$ be a set of directions with $\dim_{\mathcal H}(D)>1$, and for each $\omega \in D$ let $V_\omega$ be an affine plane perpendicular to $\omega$. It follows from results of Mitsis \cite{MitsisPlane} that the union
$\bigcup\limits_{\omega \in D} V_\omega$ has positive Lebesgue measure.
\end{remark}

\subsection{Compression and the Role of Transversality}

In higher co-dimension, related compression phenomena also arise for curved families. Bourgain constructed a curved Kakeya example in $\mathbb R^3$ in which a family of curves with all relevant directions represented is nevertheless confined to a two-dimensional set; see \cite{Bourgain1991Oscillatory}. This example illustrates that curvature alone does not prevent concentration into lower-dimensional sets in variable-parameter settings.

\begin{figure}[h]
\centering

\begin{tikzpicture}[scale=0.9]

\draw (-5.5,-3.2) rectangle (-0.5,3.2);

\node at (-3,3.7) {non-compressive regime};

\draw[thick] (-5,-2.5) .. controls (-4,-1) and (-3,0) .. (-1,2.5);
\draw[thick] (-5,-1.5) .. controls (-4,-0.2) and (-2.8,1) .. (-1,1.8);
\draw[thick] (-5,-0.5) .. controls (-4,0.5) and (-2.5,1.5) .. (-1,1);
\draw[thick] (-5,0.5) .. controls (-4,1.5) and (-2.5,2.2) .. (-1,0);
\draw[thick] (-5,1.5) .. controls (-4,2.2) and (-3,2.5) .. (-1,-1);
\draw[thick] (-5,2.5) .. controls (-4,2.8) and (-3,2) .. (-1,-2);

\node at (-3,-3.8)
{\small non-compressive behavior}; 

\draw (0.5,-3.2) rectangle (5.5,3.2);

\node at (3,3.7) {compressive regime};

\fill[gray!20] (2.2,-3.2) rectangle (3.8,3.2);

\draw[thick] (1,-2.5) .. controls (2,-1) and (3,0) .. (5,2.5);
\draw[thick] (1,-1.8) .. controls (2,-0.7) and (3,0.2) .. (5,1.8);
\draw[thick] (1,-1) .. controls (2,-0.2) and (3,0.3) .. (5,1);
\draw[thick] (1,-0.2) .. controls (2,0.2) and (3,0.4) .. (5,0.2);
\draw[thick] (1,0.6) .. controls (2,0.7) and (3,0.8) .. (5,-0.6);
\draw[thick] (1,1.4) .. controls (2,1.2) and (3,1) .. (5,-1.4);
\draw[thick] (1,2.2) .. controls (2,1.7) and (3,1.2) .. (5,-2.2);

\node at (3,-3.8)
{\small compressive behavior}; 

\end{tikzpicture}

\caption{A schematic illustration of the geometric dichotomy discussed in the text. In the non-compressive regime, curvature and transversality force the family to spread through ambient space. In the compressive regime, Kakeya-type alignment phenomena allow large curved families to concentrate near lower-dimensional sets.}
\end{figure}
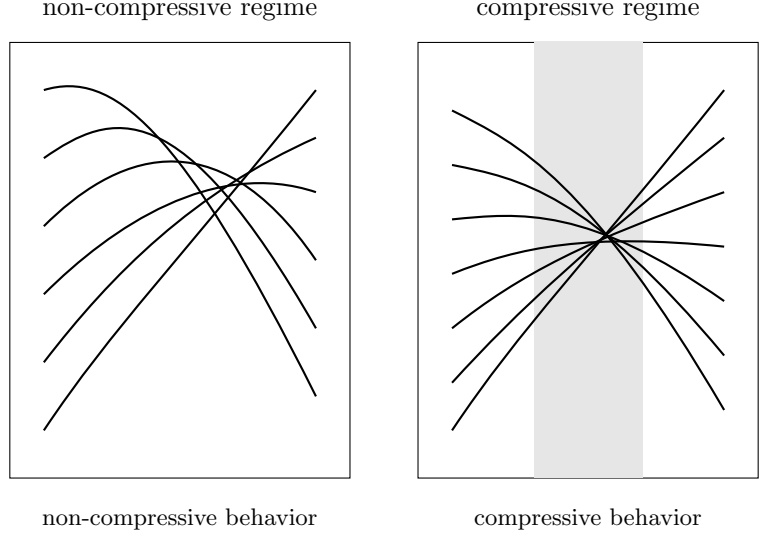

The examples discussed above suggest a broader geometric dichotomy governing smooth variable coefficient families.

Let $\phi(x,\xi)$ be a Hörmander-type phase function, and let $G(x,\xi)$ denote the associated Gauss map introduced in \cite{GWZ,Nadjimzadah}. Guo, Wang, and Zhang \cite{GWZ} introduced the condition
$$
(G(x,\xi)\cdot \nabla_x)^2\nabla_\xi^2\phi(x,\xi)
=
\lambda(x,\xi)
(G(x,\xi)\cdot \nabla_x)\nabla_\xi^2\phi(x,\xi),
$$
for some scalar function $\lambda(x,\xi)$, now known as Bourgain's condition. They conjectured that this condition precisely characterizes the Hörmander-type oscillatory integral operators satisfying the expected restriction-range estimates.

The significance of Bourgain's condition is closely tied to curved Kakeya compression phenomena. Guo, Wang, and Zhang showed that failure of Bourgain's condition is closely tied to Bourgain-type compression phenomena for curved Kakeya configurations. More recently, Nadjimzadah \cite{Nadjimzadah} showed that when Bourgain's condition holds, the associated curved tube geometry is, after suitable local changes of variables and at sufficiently small scales, modeled on the geometry of straight Kakeya tubes.

From the perspective of the present paper, this suggests the following heuristic principle. Curvature alone is not the decisive mechanism governing positivity and non-compression phenomena. Rather, the key issue is whether the associated geometric family admits Kakeya-type concentration.

For $(n-1)$-parameter families of curves in ${\mathbb R}^n$, one may conjecturally divide the smooth variable coefficient setting into two regimes:
\begin{enumerate}[(i)]
    \item the compressive regime, where Bourgain's condition fails and curved Kakeya-type concentration phenomena may occur;
    \item the non-compressive regime, where Bourgain's condition holds and the geometry is microlocally governed by the same incidence structure as the classical straight Kakeya problem.
\end{enumerate}

Under this interpretation, Bourgain's condition plays the role of a geometric rigidity condition separating genuinely curved compressive behavior from geometries which are locally equivalent to the straight model.

The results of the present paper fit naturally into this philosophy. In the hypersurface setting considered here, rotational curvature provides the Fourier integral operator smoothing mechanism, but additional geometric control is still required in the variable-level case. The Mitsis-type intersection condition introduced above may therefore be viewed as a non-compression hypothesis preventing the selected family of hypersurfaces from developing Kakeya-type alignment phenomena.

\section{Higher Co-dimension Families}

In this section we extend the positivity result to families of submanifolds of co-dimension $m \ge 1$. For each $x \in \mathbb R^d$, consider
$$
\Sigma_{x,t}=\{y \in \mathbb R^d : \phi_l(x,y)=t_l,\ 1 \le l \le m\},
$$
which, under suitable hypotheses, is a smooth submanifold of dimension $d-m$. Thus $\{\Sigma_{x,t}\}_{x\in E}$ is a family of submanifolds indexed by $E$.

Let $\vec{\phi}=(\phi_1,\dots,\phi_m)$ be a smooth function  from $\mathbb R^d \times \mathbb R^d$ to $\mathbb R^m$.

\begin{definition}[Nondegenerate System]{}
We say that the system $\vec{\phi}=(\phi_1,\dots,\phi_m)$ is nondegenerate at level $t$ if the following hold:
\begin{enumerate}[(i)]
    \item For each $x$ in the support under consideration, the set
    $$\Sigma_{x,t}=\{y\in \mathbb R^d:\phi_l(x,y)=t_l,\ 1\le l\le m\}$$
    is a smooth submanifold of co-dimension $m$.
    \item The associated generalized Radon transform
    $$Rf(x)=\int_{\Sigma_{x,t}} f(y)\psi(x,y)\,d\sigma_{x,t}(y)$$
    is a Fourier integral operator whose canonical relation is locally a canonical graph.
    \item The operator $R$ satisfies the $L^2$ Sobolev estimate
    $$
    \|R P_j f\|_2 \lesssim 2^{-j\frac{d-m}{2}} \|P_j f\|_2.
    $$
\end{enumerate}
\end{definition}


Under the nondegeneracy hypothesis above, the operator $R$ satisfies the $L^2$ Sobolev estimate stated in the definition. By standard $L^2$ Sobolev estimates for such operators (see \cite[Chapter VIII]{Sogge}), we have for frequency localized pieces
$$
\|R P_j f\|_2 \lesssim 2^{-j\frac{d-m}{2}} \|P_j f\|_2.
$$

\begin{theorem}
Suppose that $\vec{\phi}$ is nondegenerate in the sense above. If  $E \subset \mathbb R^d$ is a compact set with $\dim_{\mathcal H}(E)>m$.Then
$$
\left|\bigcup_{x \in E} \Sigma_{x,t}\right|>0.
$$
\end{theorem}

\begin{proof}
The proof is essentially the same as the proof of Theorem \ref{thm:fixed level}.

Let $\mu$ be a Frostman measure supported on $E$ with exponent $a>m$, so that
$$
\mu(B(x,r)) \lesssim  r^a.
$$

Define the incidence measure $\nu$ by
$$
\int f(y)\, d\nu(y)=\int Rf(x)\, d\mu(x).
$$

As in the hypersurface case, $\nu$ is a finite Radon measure supported on
$$
\bigcup_{x \in E} \Sigma_{x,t}.
$$

Let $\rho \in C_c^\infty(\mathbb R^d)$ be non-negative, radial, supported in the unit ball, and satisfy $\|\rho\|_1=1$.

Define
$$
\rho_\epsilon(x)=\epsilon^{-d}\rho(x/\epsilon),
$$
and
$$
\lambda_\epsilon=\nu * \rho_\epsilon.
$$

Then
$$
\|\lambda_\epsilon\|_2^2
=
\int (\nu * \rho_\epsilon)(y)(\nu * \rho_\epsilon)(y)\,dy
=
\int (\nu * \rho_\epsilon * \rho_\epsilon)(y)\, d\nu(y).
$$

This identity follows from the symmetry of $\rho_\epsilon$.

Define
$$
g_\epsilon=\nu * \rho_\epsilon * \rho_\epsilon.
$$

By the definition of $\nu$, we have
$$
\|\lambda_\epsilon\|_2^2
=
\int Rg_\epsilon(x)\, d\mu(x).
$$

We now perform a Littlewood--Paley decomposition:
$$
g_\epsilon=\sum_j g_\epsilon^j,\quad \mu=\sum_k \mu^k.
$$

Thus
$$
\|\lambda_\epsilon\|_2^2
=
\sum_{j,k} \int Rg_\epsilon^j(x)\, d\mu^k(x).
$$

We split the sum into diagonal and off-diagonal parts.

\medskip

If $|j-k|\le K$, then by Cauchy--Schwarz,
$$
\left| \int Rg_\epsilon^j\, d\mu^k \right|
\le \|R g_\epsilon^j\|_2 \|\mu^k\|_2.
$$

Using the Fourier integral operator estimate,
$$
\|R g_\epsilon^j\|_2 \lesssim  2^{-j\frac{d-m}{2}} \|g_\epsilon^j\|_2,
$$
and the Frostman estimate proved earlier,
$$
\|\mu^k\|_2 \lesssim 2^{k\frac{d-a}{2}},
$$
we obtain
$$
\left| \int Rg_\epsilon^j\, d\mu^k \right|
\lesssim  2^{-j\frac{d-m}{2}} 2^{k\frac{d-a}{2}} \|g_\epsilon^j\|_2.
$$

If $|j-k|\le K$, this yields
$$
\left| \int Rg_\epsilon^j\, d\mu^k \right|\lesssim  2^{j\frac{m-a}{2}} \|g_\epsilon^j\|_2.
$$

Since $a>m$, we have
$$
\sum_j 2^{j(m-a)}<\infty,
$$
and therefore
$$
\sum_{|j-k|\le K}
\left| \int Rg_\epsilon^j\, d\mu^k \right|
\lesssim  \|g_\epsilon\|_2.
$$

\medskip

If $|j-k|>K$, we use almost-orthogonality for Fourier integral operators whose canonical relation is locally a graph; see \cite[Chapter VIII]{Sogge}. This gives, for every $N>0$,
$$
\left| \int Rg_\epsilon^j\, d\mu^k \right|
\le C_N 2^{-N\max(j,k)} \|g_\epsilon^j\|_2 \|\mu^k\|_2.
$$

Using the Frostman estimate and summing in $j,k$, we obtain
$$
\sum_{|j-k|>K}
\left| \int Rg_\epsilon^j\, d\mu^k \right|
\lesssim \|g_\epsilon\|_2.
$$

\medskip

Combining the two parts,
$$
\|\lambda_\epsilon\|_2^2 \lesssim\|g_\epsilon\|_2.
$$

Since $g_\epsilon=\lambda_\epsilon * \rho_\epsilon$ and $\rho_\epsilon \in L^1$ with bounded norm, Young's inequality gives
$$
\|g_\epsilon\|_2 \le  \|\lambda_\epsilon\|_2.
$$

Thus
$$
\|\lambda_\epsilon\|_2^2 \lesssim \|\lambda_\epsilon\|_2,
$$
and hence $\|\lambda_\epsilon\|_2$ is uniformly bounded.

By weak compactness in $L^2$, there exists a subsequence $\epsilon_k \to 0$ and a function $h \in L^2(\mathbb R^d)$ such that $\lambda_{\epsilon_k}$ converges weakly in $L^2$ to $h$. 

However, since $\lambda_\epsilon = \nu * \rho_\epsilon$, we have $\lambda_\epsilon \to \nu$ in the sense of distributions. For every $f \in C_c(\mathbb R^d)$,
$$
\int f(y)\lambda_{\epsilon_k}(y)\,dy \to \int f(y)\,d\nu(y).
$$

By weak convergence in $L^2$, we also have
$$
\int f(y)\lambda_{\epsilon_k}(y)\,dy \to \int f(y)h(y)\,dy.
$$

Therefore
$$
\int f(y)h(y)\,dy = \int f(y)\,d\nu(y)
$$
for all $f \in C_c(\mathbb R^d)$, and hence $d\nu = h(y)\,dy$.

Since $\nu$ is nonzero, its support has positive Lebesgue measure, and the proof is complete.
\end{proof}

\vskip.125in 

\begin{example}[Translated curves in $\mathbb R^3$]
Let $d=3$ and $m=2$. Let
$$
\Gamma=\{(s,\gamma_1(s),\gamma_2(s)):s\in I\}
$$
be a smooth curve in $\mathbb R^3$ whose curvature and torsion do not vanish on $I$. For each $x\in \mathbb R^3$, consider the translated curve
$$
\Sigma_x=x+\Gamma.
$$
Equivalently, after localizing to a coordinate patch, this family can be written in the form
$$
\Sigma_x=\{y\in\mathbb R^3:\phi_1(x,y)=0,\ \phi_2(x,y)=0\},
$$
where
$$
\phi_1(x,y)=(y_2-x_2)-\gamma_1(y_1-x_1),
$$
and
$$
\phi_2(x,y)=(y_3-x_3)-\gamma_2(y_1-x_1).
$$

The associated averaging operator is convolution with arclength measure on $\Gamma$:
$$
Rf(x)=\int_I f(x+(s,\gamma_1(s),\gamma_2(s)))\chi(s)\,ds,
$$
where $\chi\in C_c^\infty(I)$. Since $\Gamma$ has nonvanishing curvature and torsion, the Fourier transform of arclength measure on $\Gamma$ satisfies
$$
|\widehat{\sigma_\Gamma}(\xi)|\lesssim |\xi|^{-\frac{1}{3}}.
$$
It follows that
$$
\|R P_j f\|_2\lesssim 2^{-\frac{j}{3}}\|P_j f\|_2.
$$

Thus this example satisfies the higher co-dimension framework with Sobolev gain $\frac{1}{3}$. It does not give the full gain $\frac{d-m}{2}=\frac{1}{2}$, but it illustrates that the hypotheses are naturally tied to the available Fourier decay of the underlying family.

This example does not satisfy the full smoothing assumption of the theorem, but the argument may be repeated with the available Sobolev gain $\frac{1}{3}$. In that case, the same computation shows that positivity holds whenever
$$
a > d - \frac{2}{3}.
$$
In dimension $d=3$, this yields the condition
$$
a > \frac{7}{3}.
$$

Thus, even in the absence of the full smoothing gain, the method still yields positivity results, with the dimensional threshold determined by the available Fourier decay of the underlying family.
\end{example}

\begin{remark}[Variable parameters and Kakeya-type phenomena in higher co-dimension]
In the higher co-dimension setting, allowing the level parameter to vary introduces phenomena closely related to Kakeya-type constructions.

To see the issue heuristically, suppose that for each $x \in E$ one selects a parameter $t(x)$ and considers the family
$$
\{\Sigma_{x,t(x)}\}.
$$
Even when each individual surface $\Sigma_{x,t}$ has nonvanishing curvature, the freedom to choose $t(x)$ may allow the resulting family to align along lower-dimensional directions, reducing the effective transverse spreading.

This behavior is analogous to classical Kakeya constructions, where families of lines with varying directions can be arranged to lie in sets of arbitrarily small measure. In the present context, variable-parameter families may similarly concentrate along thin sets unless sufficient smoothing or transversality is available.

The maximal operator $\mathcal{M}R_I$ captures the worst-case behavior over all choices of $t$, and the loss of one derivative in the Sobolev estimate reflects precisely this Kakeya-type obstruction. As a consequence, the threshold in the variable-level setting increases from $\dim_{\mathcal H}(E)>m$ in the fixed-level case to $\dim_{\mathcal H}(E)>m+1$ (or $\dim_{\mathcal H}(E)>2$ when $m=1$).
\end{remark}


\section{The Rectifiable Endpoint}

We also record the endpoint case when $E$ is rectifiable.

\begin{theorem}\label{thm: critical fixed radius}
Suppose $\phi\in C^\infty(\mathbb R^d\times\mathbb R^d),d\ge 2$ satisfies the rotational curvature condition at $1$. If $E \subset \mathbb R^d$ is a compact $1$-rectifiable set with $\mathcal H^1(E)>0$,
then
\begin{equation}
    \label{big union}
    \left| \bigcup_{x \in E} \{y:\phi(x,y)=1\} \right|>0.
\end{equation}
\end{theorem}

\begin{proof}
Since $E$ is $1$-rectifiable and $\mathcal H^1(E)>0$, there exist Lipschitz maps
$$
\gamma_n:I_n\to \mathbb R^d
$$
such that
$$
\mathcal H^1\left(E\setminus \bigcup_{n=1}^{\infty}\gamma_n(I_n)\right)=0.
$$
Hence, for some $n$, 
$$
\mathcal H^1(E\cap \gamma_n(I_n))>0.
$$
Write $\gamma=\gamma_n$ and $I=I_n$.

By the area formula for Lipschitz maps from $\mathbb R$ to $\mathbb R^d$,
$$
\int_{\gamma^{-1}(E)} |\gamma'(t)|\,dt
=
\int_E N(\gamma,\gamma^{-1}(E),x)\,d\mathcal H^1(x).
$$
where $N(\gamma,\gamma^{-1}(E),x)$ denotes the number of points in $\gamma^{-1}(x)\cap \gamma^{-1}(E)$. 

Since $\mathcal H^1(E\cap \gamma(I))>0$, after replacing $E$ by
$E\cap \gamma(I)$ the right-hand side is strictly positive. Hence
$$
\int_{\gamma^{-1}(E)} |\gamma'(t)|\,dt>0.
$$
It follows that $|\gamma'(t)|>0$ on a measurable subset
$A\subset \gamma^{-1}(E)$ of positive Lebesgue measure.

Thus
$$
\gamma(A)\subset E
$$
and
$$
|\gamma'(t)|>0
$$
for almost every $t\in A$.

Since $\gamma'$ is measurable and nonzero on a subset of $A$ of positive measure, there exists a nonzero vector $v_0$ and a measurable set $B\subset A$ of positive measure such that
$$
|\gamma'(t)-v_0|<\eta
$$
for almost every $t\in B$, where $\eta>0$ will be chosen sufficiently small.

Choose a Lebesgue density point $t_0$ of $B$ and set
$$
x_0=\gamma(t_0).
$$

We now come to the key step in the proof, which shows that the rotational curvature condition prevents the family of hypersurfaces from clustering, meaning that the surfaces do not move tangentially as $x$ moves along the Lipschitz curve.

We first claim that there exists $y_0$ such that
$$
\phi(x_0,y_0)=1
$$
and
$$
\nabla_x\phi(x_0,y_0)\cdot v_0\neq 0.
$$
Indeed, suppose otherwise. Then
$$
\nabla_x\phi(x_0,y)\cdot v_0=0
$$
for every $y$ satisfying $\phi(x_0,y)=1$.

Differentiating this identity tangentially in $y$ along the hypersurface
$$
\{y:\phi(x_0,y)=1\},
$$
we obtain
$$
v_0\cdot \partial^2_{xy}\phi(x_0,y)w=0
$$
for every tangent vector $w$ satisfying
$$
\nabla_y\phi(x_0,y)\cdot w=0.
$$
It follows that
$$
\partial^2_{xy}\phi(x_0,y)v_0
$$
is parallel to
$$
(\nabla_y\phi(x_0,y))^T.
$$
Thus, for some scalar $\lambda$,
$$
\partial^2_{xy}\phi(x_0,y)v_0=\lambda (\nabla_y\phi(x_0,y))^T.
$$
Together with
$$
\nabla_x\phi(x_0,y)\cdot v_0=0,
$$
this gives
$$
\begin{pmatrix}
0 & \nabla_x \phi \\
-(\nabla_y \phi)^T & \partial^2_{xy}\phi
\end{pmatrix}
\begin{pmatrix}
\lambda \\
v_0
\end{pmatrix}
=0
$$
at $(x_0,y)$. Since $v_0\neq 0$, this contradicts the rotational curvature hypothesis. The claim follows.

Choose such a point $y_0$. Since the rotational curvature matrix is nonsingular at $(x_0,y_0)$, in particular $\nabla_y\phi(x_0,y_0)\neq 0$. Therefore the implicit function theorem gives a smooth parametrization
$\Phi(x,u)$ defined for $x$ near $x_0$ and $u$ near some $u_0\in \mathbb R^{d-1}$, such that
$$
\Phi(x_0,u_0)=y_0
$$
and
$$
\phi(x,\Phi(x,u))=1.
$$
Moreover, the vectors
$$
\partial_{u_1}\Phi(x_0,u_0),\dots,\partial_{u_{d-1}}\Phi(x_0,u_0)
$$
span the tangent space of the hypersurface
$$
\{y:\phi(x_0,y)=1\}
$$
at $y_0$.

Differentiating
$$
\phi(x,\Phi(x,u))=1
$$
in the $x$ direction $v_0$, at $(x_0,u_0)$, gives
$$
\nabla_x\phi(x_0,y_0)\cdot v_0
+
\nabla_y\phi(x_0,y_0)\cdot D_x\Phi(x_0,u_0)v_0
=0.
$$
Since
$$
\nabla_x\phi(x_0,y_0)\cdot v_0\neq 0,
$$
we obtain
$$
\nabla_y\phi(x_0,y_0)\cdot D_x\Phi(x_0,u_0)v_0\neq 0.
$$
Therefore $D_x\Phi(x_0,u_0)v_0$ is transverse to the hypersurface
$$
\{y:\phi(x_0,y)=1\}.
$$
It follows that the $d$ vectors
$$
D_x\Phi(x_0,u_0)v_0,
\partial_{u_1}\Phi(x_0,u_0),\dots,\partial_{u_{d-1}}\Phi(x_0,u_0)
$$
span $\mathbb R^d$.

Consequently, by continuity, if $\eta>0$ is chosen sufficiently small, then there exist neighborhoods $V$ of $x_0$ and $U$ of $u_0$ such that
$$
\left|
\det
\left(
D_x\Phi(x,u)w,
\partial_{u_1}\Phi(x,u),\dots,\partial_{u_{d-1}}\Phi(x,u)
\right)
\right|
\ge c>0
$$
whenever
$$
x\in V,\qquad u\in U,\qquad |w-v_0|<\eta.
$$

Since $\gamma$ is continuous and $t_0$ is a density point of $B$, we may choose a sufficiently small interval $J\subset I$ containing $t_0$ such that
$$
\gamma(t)\in V
$$
for all $t\in J$, and
$$
|B\cap J|>0.
$$

Define
$$
F(t,u)=\Phi(\gamma(t),u)
$$
for $(t,u)\in (B\cap J)\times U$.

\begin{figure}[h]
\centering
\begin{tikzpicture}[scale=1]


\pgfmathsetmacro{\xzero}{-0.8}
\pgfmathsetmacro{\yzero}{0.25 + 0.35*(\xzero) - 0.08*(\xzero)^3}
\pgfmathsetmacro{\mzero}{0.35 - 0.24*(\xzero)^2} 

\pgfmathsetmacro{\xu}{0}
\pgfmathsetmacro{\yu}{1.5 - 0.12*(\xu)^2}

\draw[thick,domain=-2:2,samples=100,smooth]
  plot (\x,{0.25 + 0.35*(\x) - 0.08*(\x)^3});
\node at (3,0.3) {$E=\gamma(t)$};

\fill (\xzero,\yzero) circle (0.05);
\node[above] at (\xzero,\yzero) {$x_0$};

\draw[->,thick]
  (\xzero,\yzero) -- ++(1,{1*\mzero});
\node at (0.45,0.15) {$v_0$};

\draw[thick,domain=-1.8:1.8,samples=100,smooth]
  plot (\x,{1.5 - 0.12*(\x)^2});
\node at (3.2,1.5) {$\{y:\phi(x_0,y)=1\}$};

\draw[->] (\xu,\yu) -- ++(-0.6,0.6);
\draw[->] (\xu,\yu) -- ++(0.6,0.4);
\node at (0.2,2.1) {$u$};

\draw[->] (\xu,\yu) -- ++(0.6,-0.6);
\node at (1.0,0.95) {$t$};

\end{tikzpicture}
\caption{Schematic of the map $F(t,u)=\Phi(\gamma(t),u)$: variation in $t$ moves the hypersurface transversely, while the $u$-variables parametrize the hypersurface itself.}
\end{figure}
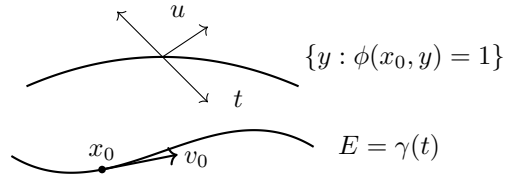

The map $F$ is Lipschitz. For almost every $t\in B\cap J$, the derivative $\gamma'(t)$ exists and satisfies
$$
|\gamma'(t)-v_0|<\eta.
$$
Therefore the Jacobian of $F$ satisfies
$$
J_F(t,u)
=
\left|
\det
\left(
D_x\Phi(\gamma(t),u)\gamma'(t),
\partial_{u_1}\Phi(\gamma(t),u),\dots,\partial_{u_{d-1}}\Phi(\gamma(t),u)
\right)
\right|
\ge c
$$
for almost every $(t,u)\in (B\cap J)\times U$.

By the area formula,
$$
\int_{(B\cap J)\times U} J_F(t,u)\,dt\,du
=
\int_{\mathbb R^d} N(F,(B\cap J)\times U,y)\,dy,
$$
where $N(F,(B\cap J)\times U,y)$ denotes the number of preimages of $y$ in $(B\cap J)\times U$.

Since
$$
|(B\cap J)\times U|>0
$$
and $J_F\ge c$ almost everywhere on this set, the left hand side is positive. Hence the image
$$
F((B\cap J)\times U)
$$
has positive Lebesgue measure.

Finally, for every $t\in B\cap J$ and every $u\in U$,
$$
\phi(\gamma(t),F(t,u))=1.
$$
Since $\gamma(t)\in E$ for $t\in B\subset A$, we have
$$
F((B\cap J)\times U)
\subset
\bigcup_{x\in E}\{y:\phi(x,y)=1\}.
$$
Therefore
$$
\left| \bigcup_{x\in E}\{y:\phi(x,y)=1\} \right|>0.
$$
This completes the proof.
\end{proof}

\begin{remark}[]{}
When $\phi(x,y) = |x-y|$, or more generally if $\phi$ is a norm associated to a convex body with nonvanishing curvature, then Theorem \ref{thm: critical fixed radius} follows as a consequence of the main results in \cite{SimonTaylor, BongersTaylor}, with the third listed author, when $d=2$ and $d>2$, respectively. 
Indeed, the quantity in \eqref{big union} is equal to the Minkowski sum $E+ S^{d-1}$, where $S^{d-1}$ is the appropriate unit sphere.
Further, this quantity is comparable to the Favard curve length, as defined in \cite{CladekDaveyTaylor, BongersTaylor, SimonTaylor}. 
See also \cite{LiTaylorTwoproj, BBMT26}, where a direct proof shows that if $E$ is $1$-rectifiable, then the Favard curve length has positive measure, and hence the conclusion of Theorem \ref{thm: critical fixed radius} holds. 
\end{remark}

\begin{remark} 
As an application of Theorem \ref{thm: critical fixed radius}, we observe that if $E\subset\R^2$ is a subset of the unit circle, $S^1$, satisfying  $\mathcal{H}^1(E)>0$, then the union of the tangent lines of $S^1$ at the points of $E$ has positive Lebesgue measure.  This follows from observing that the dot product, $\phi(x,y)= x\cdot y$ serves as an example of a function that satisfies the hypothesis of Theorem \ref{thm: critical fixed radius}. 
\end{remark}

\subsection{Comments on Variable Levels}

In the special case of circles in the plane, the fixed-radius setting is already covered by the results mentioned in the previous remark. It is natural to ask whether the rectifiable endpoint continues to hold when the radius is allowed to vary.

Let $E \subset \mathbb R^2$ be $1$-rectifiable with $\mathcal H^1(E)>0$, and let $r:\R^2\to (0,\infty)$. Consider
$$
B=\bigcup_{x\in E}\{y:|y-x|=r(x)\}.
$$
If $r$ has sufficient regularity along a parametrization of $E$, for example if $r\circ \gamma$ is Lipschitz for a Lipschitz parametrization $\gamma$ of a subset of $E$ of positive measure, then the argument used in the proof of the rectifiable endpoint adapts and yields $|B|>0$. This indicates that in the presence of curvature, one may recover positivity under mild structural assumptions on the parameter dependence.

However, without regularity assumptions on $r$, one cannot expect such a conclusion. A construction of Talagrand shows that there exists a set of Lebesgue measure zero in the plane that contains a circle centered at every point of a line \cite{Talagrand}. This indicates that the dimension $1$ variable-radius case is fundamentally different from the fixed-radius setting.

\section{Riemannian Manifold Analogues} 

We also record a Riemannian analogue of the main result. Let $(M,g)$ be a compact Riemannian manifold without boundary, and let $\rho(x,y)$ denote the Riemannian distance function. For radii below the injectivity radius, the geodesic spheres
$$
\{y\in M:\rho(x,y)=t\}
$$
form a variable coefficient family satisfying the same microlocal structure as in the Euclidean setting. The associated averaging operators are Fourier integral operators satisfying the same local $L^2$ smoothing estimates; see \cite{Sogge,Stein}. Consequently, the local arguments developed in this paper admit corresponding analogues for unions of geodesic spheres on compact Riemannian manifolds, at least below the injectivity radius and away from conjugate points, with the same formal Hausdorff-dimension thresholds. This supports the viewpoint that the underlying positivity mechanism is fundamentally microlocal and geometric rather than specifically Euclidean.

In this section, we give a sketch of how some of the ideas of this paper transfer to the Riemannian manifold setting. A more detailed exploration will be conducted in the sequel. 

Let $M$ be a compact $d$-dimensional Riemannian manifold without boundary, and let $\rho(x,y)$ denote the Riemannian distance. For $0<t<\operatorname{inj}(M)$, define the averaging operator
$$
A_t f(x)=\int_{\rho(x,y)=t} f(y)\,d\sigma_{x,t}(y),
$$
where $d\sigma_{x,t}$ is the induced surface measure on the geodesic sphere.

\begin{theorem}
Let $M$ be a compact $d$-dimensional Riemannian manifold without boundary and $0<t_0<\operatorname{inj}(M)$. If $E\subset M$ is a compact set with $\dim_{\mathcal H}(E)>1,$
then
$$
\left| \bigcup_{x \in E} \{y \in M : \rho(x,y)=t_0\} \right|>0.
$$
\end{theorem}

\begin{remark} The conclusion is independent of the global geometry of $M$. No curvature assumptions on $M$ are required beyond smoothness, and the result applies equally to manifolds of positive, negative, or variable sectional curvature. The mechanism is entirely local, depending only on the microlocal structure of the geodesic sphere averaging operator and its associated Fourier integral representation. \end{remark}

\begin{proof}
Since $M$ is compact, it may be covered by finitely many coordinate charts. 

Let $U \subset M$ be a coordinate neighborhood. For $t_0<\operatorname{inj}(M)$, the set
$$
\{y : \rho(x,y)=t_0\}
$$
is a smooth hypersurface depending smoothly on $x$, as long as $x$ and $y$ remain in $U$. Moreover, the Riemannian distance function $\rho(x,y)$ is smooth away from the diagonal, and in local coordinates the associated phase function generates a canonical relation satisfying the standard nondegeneracy conditions for Fourier integral operators; see \cite{Sogge,Stein}.

Let $\{\chi_\alpha\}$ be a partition of unity subordinate to the atlas, and write
$$
A_{t_0} f(x)=\sum_{\alpha,\beta} A_{t_0}^{\alpha,\beta} f(x),
$$
where
$$
A_{t_0}^{\alpha,\beta} f(x)=\chi_\alpha(x)\int_{\rho(x,y)=t_0} f(y)\chi_\beta(y)\,d\sigma_{x,t_0}(y).
$$
Each operator $A_{t_0}^{\alpha,\beta}$ may be written in local coordinates as a Fourier integral operator associated to the phase function
$$
\phi_{\alpha,\beta}(x,y)=\rho(x,y),
$$
with a smooth amplitude compactly supported in the coordinate patch. 

It follows from the theory of Fourier integral operators on manifolds that each $A_{t_0}^{\alpha,\beta}$ gains $\frac{d-1}{2}$ derivatives on $L^2$; see \cite{Sogge,Stein}. More precisely, if $P_j$ denotes a spectral localization operator to frequencies of size $\sim 2^j$ defined via the functional calculus of the Laplace--Beltrami operator, then
$$
\|A_{t_0}^{\alpha,\beta} P_j f\|_2 \lesssim  2^{-j\frac{d-1}{2}} \|P_j f\|_2.
$$

Let $\mu$ be as above. Define the incidence measure $\nu$ by
$$
\int f(y)\,d\nu(y)=\int A_{t_0}f(x)\,d\mu(x).
$$
Using the partition of unity, this reduces to finitely many terms corresponding to the operators $A_{t_0}^{\alpha,\beta}$.

The remainder of the argument proceeds as in the proof of the main theorem, with the Littlewood--Paley decomposition replaced by a spectral decomposition associated to the Laplace--Beltrami operator; see \cite{Sogge}. The smoothing estimate above yields the same bounds for the frequency-localized pieces, and therefore the argument shows that a mollification of $\nu$ is uniformly bounded in $L^2$. It follows that $\nu$ is absolutely continuous with an $L^2$ density.

Since $\nu$ is nonzero, its support has positive Riemannian volume. This support is contained in
$$
\bigcup_{x \in E} \{y \in M : \rho(x,y)=t_0\},
$$
which therefore has positive measure. This shows that the positivity mechanism is stable under passage to compact manifolds and depends only on local canonical geometry.
\end{proof}

\begin{remark}[Compression phenomena in the Heisenberg group]
Related compression phenomena also arise naturally in sub-Riemannian geometry, particularly in the Heisenberg group ${\mathbb H}^n$. In that setting, horizontal curves and geodesics exhibit strong Kakeya-type behavior, and the existence and structure of Heisenberg Besicovitch sets have been studied extensively in connection with the Kakeya problem; see, for example, \cite{Heisenberg,KatzTaoKakeya} and the references therein.

From the perspective of the present paper, these examples provide further evidence that curvature or noncommutativity alone does not prevent compression phenomena. Rather, the decisive issue is the incidence geometry and the possibility of Kakeya-type alignment. This is closely related to the broader philosophy emerging from the work of Bourgain, Katz, Tao, and many others on Euclidean and curved Kakeya problems.

It would therefore be interesting to understand whether an analogue of the intersection mechanism introduced above can be formulated in the sub-Riemannian setting, and whether this leads to corresponding positivity or non-compression results for families of Heisenberg geodesics or other horizontal submanifolds.
\end{remark}

\section{Failure of Interior}
It is natural to ask whether the results and methods on positivity of Lebesgue measure imply nonempty interior. However, even in the fixed-radius setting, positivity of the Lebesgue measure does not guarantee any interior regularity.

The following lemma shows that, even for compact rectifiable sets of centers, the union of unit circles may have positive Lebesgue measure while still failing to contain any open ball.

\begin{lemma}
There exists a compact rectifiable $1$-set $E \subset {\mathbb R}^2$ such that
$$
\bigcup_{x \in E} \{y : |y-x|=1\}
$$
has positive Lebesgue measure but empty interior.
\end{lemma}

\begin{proof}
Let $F \subset [0,1]$ be a fat Cantor set. Thus $F$ is compact, has positive one-dimensional Lebesgue measure, and has empty interior. Set
$$
E=F \times \{0\}.
$$
Then $E$ is compact and rectifiable.

Let
$$
U=\bigcup_{x \in E} \{y : |y-x|=1\}.
$$
We can write
$$
U=\{(t+\cos \theta,\sin \theta): t \in F,\theta \in [0,2\pi]\}.
$$

Define
$$
\Phi(t,\theta)=(t+\cos \theta,\sin \theta)
$$
for $t \in F$ and $\theta \in [\pi/6,\pi/3]$. On this set the map is injective, and its Jacobian determinant is $|\cos \theta|$, which is bounded below by a positive constant. Since $F$ has positive one-dimensional Lebesgue measure, the product set $F \times [\pi/6,\pi/3]$ has positive two-dimensional Lebesgue measure. By the area formula, it follows that $\Phi(F \times [\pi/6,\pi/3])$ has positive Lebesgue measure. Hence $|U|>0$.

It remains to show that $U$ has empty interior. For $y \in \mathbb R$, let
$$
U_y=\{x \in \mathbb R : (x,y)\in U\}.
$$
If $|y|>1$, then $U_y$ is empty. If $|y|\le 1$, then any point $(x,y)\in U$ satisfies
$$
x=t+\sqrt{1-y^2}
$$
or
$$
x=t-\sqrt{1-y^2}
$$
for some $t \in F$. Hence
$$
U_y \subset (F+\sqrt{1-y^2}) \cup (F-\sqrt{1-y^2}).
$$
Since $F$ is compact and has empty interior, each translate of $F$ is compact and has empty interior. Therefore the finite union on the right has empty interior in $\mathbb R$. Thus every horizontal slice $U_y$ has empty interior.

\begin{figure}[h]
\centering
\begin{tikzpicture}[scale=1]

\draw[->] (-3,0) -- (3,0);

\node at (3.2,0) {$x$};
\node at (-3.2,0) {$U_y$};

\foreach \x in {-2.5,-2.1,-1.8,-1.3,-1.0,-0.6}
{
\fill (\x,0) circle (0.05);
}

\foreach \x in {0.6,1.0,1.3,1.8,2.1,2.5}
{
\fill (\x,0) circle (0.05);
}

\node at (-1.5,0.5) {$F+\sqrt{1-y^2}$};
\node at (1.5,0.5) {$F-\sqrt{1-y^2}$};

\node at (0,-1) {finite union of Cantor-type sets: no intervals};

\end{tikzpicture}
\caption{A horizontal slice $U_y$ is contained in a finite union of translates of a Cantor set, and therefore has empty interior.}
\end{figure}
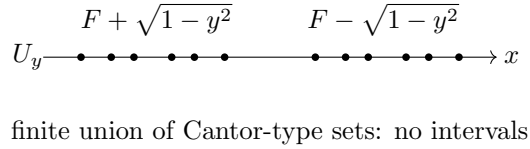

If $U$ contained a nonempty open ball, then for some $y$ the horizontal slice of that ball would contain a nonempty open interval. This interval would be contained in $U_y$, contradicting the fact that $U_y$ has empty interior. Therefore $U$ has empty interior.
\end{proof}

The preceding lemma shows that even for compact rectifiable sets of centers, positivity of Lebesgue measure of the associated union of circles does not imply the presence of interior.
\vskip.5in

\bibliography{refs}
\bibliographystyle{abbrv}

\end{document}